\documentclass[12pt,reqno]{amsart}
\usepackage[numbers,sort&compress]{natbib}
\usepackage{amsfonts,bbm}
\usepackage{amssymb,color}
\usepackage{amssymb}
\usepackage{fancyhdr}
\usepackage[titletoc]{appendix}
\usepackage{enumitem}

\usepackage[colorlinks=true,backref]{hyperref}
\hypersetup{urlcolor=blue, citecolor=blue, linkcolor=black}

\usepackage[margin=3cm, a4paper]{geometry}

\newtheorem{thm}{Theorem}[section]
\newtheorem{cor}[thm]{Corollary}
\newtheorem{lem}[thm]{Lemma}
\newtheorem{prop}[thm]{Proposition}
\newtheorem{defn}[thm]{ \bf{Definition}}
\theoremstyle{remark}

\newcommand{\EQ}[1]{\begin{align*}\begin{split} #1 \end{split}\end{align*}}
\newcommand{\EQn}[1]{\begin{align}\begin{split} #1 \end{split}\end{align}}
\newcommand{\EQnn}[1]{\begin{align} #1 \end{align}}

\newcommand{\enu}[1]{\begin{enumerate} #1 \end{enumerate}}

\newcommand{\Del}[1]{}

\def\norm#1{\left\|#1\right\|}

\def\normb#1{\big\|#1\big\|}

\def\abs#1{\left|#1\right|}

\def\absb#1{\big|#1\big|}

\def\brk#1{\left(#1\right)}
\def\brko#1{(#1)}
\def\brkb#1{\big(#1\big)}

\def\fbrk#1{\left\lbrace#1\right\rbrace}
\def\fbrko#1{\lbrace#1\rbrace}
\def\fbrkb#1{\big\lbrace#1\big\rbrace}

\def\jb#1{\langle#1\rangle}

\def\wt#1{\widetilde{#1}}
\def\wh#1{\widehat{#1}}
\def\wb#1{\overline{#1}}

\def\pd{\partial}

\newcommand{\ra}{{\rightarrow}}
\newcommand{\hra}{{\hookrightarrow}}

\def\loe{\leqslant}
\def\goe{\geqslant}
\def\lsm{\lesssim}
\def\gsm{\gtrsim}

\newcommand{\N}{{\mathbb N}}
\newcommand{\T}{{\mathbb T}}
\newcommand{\R}{{\mathbb R}}
\newcommand{\C}{{\mathbb C}}

\newcommand{\PP}{{\mathbb P}}

\newcommand{\F}{{\mathcal{F}}}
\newcommand{\A}{{\mathcal{A}}}
\newcommand{\B}{{\mathcal{B}}}

\newcommand{\Sch}{{\mathcal{S}}}

\newcommand{\Q}{{\mathcal{Q}}}

\newcommand{\supp}{{\mbox{supp}\ }}

\def\dx{\mathrm{\ d} x}
\def\dy{\mathrm{\ d} y}
\def\ds{\mathrm{\ d} s}
\def\dt{\mathrm{\ d} t}

\newcommand{\re}{{\mathrm{Re}}}
\newcommand{\im}{{\mathrm{Im}}}

\def\ep{\varepsilon}
\def\al{\alpha}

\def\Om{\Omega}
\def\om{\omega}
\def\ph{\varphi}
\def\th{\theta}

\def\de{\delta}
\def\De{\Delta}

\def\la{\lambda}
\def\ga{\gamma}

\newcommand{\I}{\infty}
\def\rev#1{\frac{1}{#1}}
\def\half#1{\frac{#1}{2}}

\def\bx{\square}

\numberwithin{equation}{section}

\begin{document}
\title[Energy critical NLS]{Almost sure scattering for the nonradial energy-critical NLS with arbitrary regularity in 3D and 4D cases}

\author{Jia Shen}
\address{(J. Shen) Center for Applied Mathematics\\
Tianjin University\\
Tianjin 300072, China}
\email{shenjia@tju.edu.cn}

\author{Avy Soffer}
\address{(A. Soffer) Rutgers University\\
	Department of Mathematics\\
	110 Frelinghuysen Rd.\\
	Piscataway, NJ, 08854, USA\\}
\email{soffer@math.rutgers.edu}
\thanks{}

\author{Yifei Wu}
\address{(Y. Wu) Center for Applied Mathematics\\
	Tianjin University\\
	Tianjin 300072, China}
\email{yerfmath@gmail.com}
\thanks{}


\date{}

\begin{abstract}\noindent
In this paper, we study the defocusing energy-critical nonlinear Schr\"odinger equations
$$
i\pd_t u + \De u =  |u|^{\frac{4}{d-2}} u.
$$
When $d=3,4$, we prove the almost sure scattering for the equations with non-radial data in $H_x^s$ for any $s\in\R$. In particular, our result does not rely on any spherical symmetry, size or regularity restrictions.
\end{abstract}

\maketitle

\tableofcontents

\section{Introduction}

In this paper, we consider the nonlinear Schr\"odinger equations (NLS):
\EQn{
	\label{eq:nls}
	\left\{ \aligned
	&i\pd_t u + \De u = \mu |u|^p u, \\
	& u(0,x) = u_0(x),
	\endaligned
	\right.
}
where $p>0$, $\mu=\pm1$, and
$u(t,x):\R\times\R^d\rightarrow \C$ is a complex-valued function. The positive sign ``$+$" in nonlinear term of \eqref{eq:nls} denotes defocusing source,   and the negative sign ``$-$" denotes the focusing one.

The equation \eqref{eq:nls} has conserved mass
\EQn{\label{NLS:mass}
	M(u(t)) :=\int_{\R^d} \abs{u(t,x)}^2 \dx=M(u_0),
}
and energy
\EQn{\label{NLS:Energy}
	E(u(t)) :=\int_{\R^d} \half 1\abs{\nabla u(t,x)}^2 \dx + \mu \int_{\R^d} \rev{p+2} \abs{u(t,x)}^{p+2} \dx=E(u_0).
}

The class of solutions to equation (\ref{eq:nls}) is invariant under the scaling
\begin{equation}\label{eqs:scaling-alpha}
u(t,x)\to u_\lambda(t,x) = \lambda^{\frac{2}{p}} u(\lambda^2 t, \lambda x) \ \ {\rm for}\ \ \lambda>0,
\end{equation}
which maps the initial data as
\EQn{
u(0)\to u_{\lambda}(0):=\lambda^{\frac{2}{p}} u_0(\lambda x) \ \ {\rm for}\ \ \lambda>0.
}
Denote
$$
s_c=\frac d2-\frac{2}{p},
$$
then the scaling  leaves  $\dot{H}^{s_{c}}$ norm invariant, that is,
\begin{eqnarray*}
	\|u(0)\|_{\dot H^{s_{c}}}=\|u_{\lambda}(0)\|_{\dot H^{s_{c}}}.
\end{eqnarray*}
This gives the scaling critical exponent $s_c$. Let
$$
2^*=\I, \mbox{ when } d=1 \mbox{ or } d=2;  \quad
2^*=\frac4{d-2}, \mbox{ when }  d\goe 3.
$$
Therefore, according to the conservation law, the equation is called mass or $L_x^2$ \textit{critical} when $p=\frac 4d$, and energy or $\dot H_x^1$ \textit{critical} when $p=\frac4{d-2}$. Moreover, when $\frac 4d<p<2^*$, we say that the equation is \textit{inter-critical}.

Let us now recall the well-posedness and scattering theory of NLS \eqref{eq:nls}. There are extensive studies about the subject, and we do not intend to mention all the results. Therefore in this paper, we mainly focus on the energy critical case. First, the equation is locally well-posed in $H_x^s$ with $s\goe 1$, and ill-posed in the case $s<1$, see \cite{Caz03book,CCT03illposed}.

The first global well-posedness and scattering result for the energy critical NLS was established by Bourgain \cite{Bou99JAMS}. Then, the defocusing case was proved by Colliander, Keel, Staffilani, Takaoka, and Tao \cite{CKSTT08Annals} in three dimensional case, and Ryckman and Visan  \cite{RV07AJM,Vis07Duke} in higher dimensional cases. For the focusing equations, Kenig and Merle \cite{KM06Invent} first studied the dynamics below the energy of ground state, and then the result is extended by Killip and Visan \cite{KV10AJM} in five and higher dimensions, Dodson \cite{Dod19ASENS} in four dimension.

Now, we turn to the probability theory of NLS. Although there are  ill-posedness results below the critical regularity for NLS due to the result of Christ, Colliander, and Tao \cite{CCT03illposed}, Bourgain \cite{Bou94CMP,Bou96CMP} first introduced a probabilistic method to study the well-posedness problem for periodic NLS for ``almost'' all the initial data in super-critical spaces. The probabilistic well-posedness result for super-critical wave equations on compact manifolds was also studied by Burq and Tzvetkov \cite{BT08inventI,BT08inventII}. There have been extensive studies about such subject since then, and we refer the readers to \cite{BOP19note} for more complete overviews.

Next, we only review the study of random data theory mainly for NLS on $\R^d$. There are several ways of randomization for the initial data. We start with the Wiener randomization, which is related to a unit-scale decomposition in frequency. For NLS, the almost sure local well-posedness, small data scattering, and ``conditional'' global well-posedness were considered in \cite{BOP15TranB,BOP19TranB,PW18lp,Bre19tunis,SSW21cubicNLS}.

The random data global well-posedness for the energy critical problems was first proved in the context of non-linear wave equations (NLW) by Pocovnicu \cite{Poc17JEMS}. See also \cite{OP16JMPA} for the 3D result. As for the NLS, Oh, Okamoto, and Pocovnicu \cite{OOP19DCDS} studied the almost sure global well-posedness in the energy critical case when $d=5,6$.

The large data almost sure scattering was first obtained by Dodson, L\"uhrmann, and Mendelson \cite{DLM20AJM} for the 4D, defocusing, energy-critical NLW with randomized radial data in $H_x^s$ for $\frac12<s<1$, using a double bootstrap argument combining the energy and Morawetz estimates. The first almost sure scattering result for NLS was given by Killip, Murphy, and Visan \cite{KMV19CPDE}. They proved the result for 4D, defocusing, energy-critical case with almost all the radial initial data with $\frac56<s<1$. This result was then improved to $\frac12<s<1$ by Dodson, L\"uhrmann, and Mendelson \cite{DLM19Adv}. While all the above scattering results concerns the energy critical case, the authors \cite{SSW21cubicNLS} and Camps \cite{Cam21NLS} independently proved the random data global well-posedness and scattering for the 3D defocusing cubic NLS, which is a typical model of inter-critical NLS.

We remark that the Wiener randomization is closely related to the modulation space introduced by Feichtinger \cite{Fei83modulation}. Such space has been applied to non-linear evolution equations before the development of Wiener randomization, dating back to the results of Wang, Zhao, Guo, and Hudzik \cite{WZG06JFA,WH07JDE}.

There are also other kinds of randomization for NLS on $\R^d$. Burq, Thomann, and Tzvetkov \cite{BTT13Fourier,BT20scattering} introduced a randomization based on the invariant measure for NLS with harmonic potential, and proved almost sure $L^2$-scattering for 1D defocusing NLS. Such randomization relies on the countable eigen-basis of the Laplacian, while the Wiener randomization directly comes from the decomposition in frequency space. In addition, Murphy \cite{Mur19ProAMS} introduced an another randomization based on the physical space unit-scale decomposition to study the almost sure wave operator problem.

Furthermore, other randomizations have been applied to the almost sure global well-posedness and scattering of energy critical models. In \cite{Bri18NLW}, Bringmann introduced a randomization based on wave packet decomposition to study the non-radial 4D NLW in $H_x^s$ with $s>\frac{11}{12}$.  
In \cite{Bri20APDE}, Bringmann introduced another randomization based on annuli decomposition, and observed that the smaller the scale of decomposition is, the more smoothing effect one can expect. Based on this, he proved that the 3D energy-critical wave equation is almost surely scattering in $H^s, s>0$ in the radial case.
The related results on non-radial energy-critical nonlinear Klein-Gordon equations were studied by Chen and Wang \cite{CW20NLW} for $d=4$ and $5$.  In the very recently, we also learn that the almost sure scattering for the 4D non-radial NLS was obtained by Spitz \cite{Spi21NLS} in $H_x^s$ with $\frac{5}{7}<s<1$, using a randomization similar to the one introduced by Burq and Krieger \cite{BK19NLW}. 

We remark that all the previous almost sure scattering results require the initial data in $H_x^s$ with $s\goe 0$. A nature question is whether the problem is almost surely well-posed in rougher space. In fact, 
the well known Gaussian white noise is almost surely in $H^s_x, s<-\frac d2$, and the Brownian motion is almost surely in $H^s,s<-\frac12$, see for instance \cite{BO11AdV, Ver11whitenoise}. Therefore, this motivates us to consider the random data problem at arbitrary regularity setting.

In this paper, we are able to prove the almost sure global well-posedness and scattering for 3D and 4D energy critical NLS with non-radial data at arbitrary regularity. 

\subsection{Definition of randomization}\label{sec:construction}
To this end, we first introduce a new version of randomization based on rescaled cubes decomposition in frequency space, inspired by \cite{Bri20APDE} and the Wiener randomization. Before stating the main result, we give some definitions.

Let $a\in\N$. First, for $N\in 2^{\N}$, we denote the cube sets 
\EQ{
O_N=&\fbrkb{\xi\in \R^d:  |\xi_j| \le N, j=1,2,\cdots,d},
}
and then we define
\EQ{
Q_N=& O_{2N}\setminus O_{N}.
}

Next, we make a further decomposition of $Q_N$.
Note that $a$ is a positive integer, we make a partition of the $Q_N$ with the essentially disjoint sub-cubes
\EQ{
\A(Q_N):=\fbrk{Q: Q\text{ is a dyadic cube with length } N^{-a}\text{ and }Q\subset Q_N}.
}
Then, we have $\sharp \A( Q_N)=(2^d-1)N^{d(a+1)}$ and  $ Q_N=\cup_{Q\in\A(Q_N)}Q$.  

Now, we can define the final decomposition
\EQ{
	\Q:=\fbrk{O_1} \cup \fbrk{Q: Q\in \A( Q_N)\text{ and } N\in2^\N}.
}
By the above construction,
\EQ{
\R^d=  O_1 \cup \brk{ \cup_{N\in 2^\N} Q_N} 
= O_1\cup\brk{\cup_{N\in 2^\N}\cup_{Q\in\A(Q_N)}Q} =\cup_{Q\in\Q}Q.
}
Therefore, $\Q$ is a countable family of essentially disjoint caps covering $\R^d$. We can renumber the cubes in $\Q$ as follows:
\EQ{
\Q=\fbrk{Q_j:j\in \N}.
}

\begin{defn}[``Narrowed" Wiener Randomization]\label{defn:randomization}
Let $d\goe 1$. Given $s\in \R$, we set the parameter $a\in\N$  such that   
\EQ{
a>\max\fbrk{3-4s,1-2s, 10}.
}
Let $\wt \psi_j \in C_0^\I(\R^d)$ be a real-valued function such that $0\loe \wt \psi_j\loe 1$ and
\EQ{
	\wt \psi_j(\xi) =\left\{ \aligned
	&1\text{, when $\xi\in Q_j$,}\\
	&\text{smooth, otherwise,}\\
	&0\text{, when $\xi\notin 2Q_j$,}
	\endaligned
	\right.
}
where we use the notation that $2Q_j$ is the cube with the same center as $Q_j$ and with $\mathrm{diam}(2Q_j)=2\mathrm{diam}(Q_j)$. Now, let
\EQ{
	\psi_j(\xi):=\frac{\wt \psi_j(\xi)}{\sum_{j'\in\N}\wt \psi_{j'}(\xi)}.
} 
Then, $\psi_j\in C_0^\I([0,+\I))$ is a real-valued function, satisfying $0\loe \psi_j\loe 1$ and for all $\xi \in\R^d$, $\sum_{j\in\N}\psi_j(\xi)=1$.

Denote the Fourier transform on $\R^d$ by $\F$. Then, for the function $f$ on $\R^d$, we define 
\EQ{
	\bx_jf=\F_\xi^{-1}\brkb{\psi_j(\xi)\F f(\xi)}.
}
Let $(\Om, \A, \PP)$ be a probability space. Let $\fbrk{g_j}_{j\in\N}$ be a sequence of zero-mean, complex-valued Gaussian random variables on $\Om$, where the real and imaginary parts of $g_j$ are independent. Then, for any function $f$, we define its randomization $f^\om$ by
\EQn{\label{eq:randomization}
	f^\om=\sum_{j\in\N} g_j(\om)\bx_j f.
}
\end{defn}

\subsection{Main result}

In the following, we use the statement ``almost every $\om\in\Om$, $PC(\om)$ holds'' to mean that 
$$
\PP\Big(\big\{\om\in \Om: PC(\omega)\> \mbox{ holds} \big\}\Big)=1.
$$
Arguing similarly as in the previous works (see Appendix B. in \cite{BT08inventI}), we can show that this randomization does not improved the regularity of $f$, namely for any $s\in\R$
\EQ{
f\notin H_x^s(\R^d)\text{, then } f^\om\notin H_x^s(\R^d)\text{ almost surely.}
}
Now, we study the defocusing, energy critical NLS with randomized initial data:
\EQn{
	\label{eq:nls-3D}
	\left\{ \aligned
	&i\pd_t u + \De u =  |u|^{\frac{4}{d-2}} u, \\
	& u(0,x) = f^\om(x).
	\endaligned
	\right.
}
Our main result is as follows: 
\begin{thm}[Global well-posedness and scattering]\label{thm:global}
Let $d=3$ or $d=4$. Given any  $s\in \R$ and $f\in H_x^s(\R^d)$. Suppose that the randomization $f^\om$ is defined in Definition \ref{defn:randomization}. Then, for almost every $\om\in \Om$, there exists a global solution $u$ of \eqref{eq:nls-3D} such that
\EQ{
	u-e^{it\De}f^\om\in C(\R;H_x^{1}(\R^d)). 
}
Moreover, the solution $u$ scatters, in the sense that there exist $u_\pm\in H_x^1(\R^d)$ such that
\EQ{
	\lim_{t\ra\pm\I}\norm{u-e^{it\De}f^\om-e^{it\De}u_\pm}_{H_x^1(\R^d)} =0.
}
\end{thm}

As shown in Definition \ref{defn:randomization}, for given $s\in\R$, we can find a randomization with the parameter $a=a(s)$, such that the theorem holds. This implies that our result has no regular restriction, and the equation could be solved globally no matter how rough the initial data is. Particularly, when $s<0$, the initial data has infinite mass and energy. Hence, a new difficulty in this situation is that  both the mass and energy conservation laws are not available.

Previously, there is no probabilistic global result for the 3D energy critical NLS with large data. As mentioned above, for the 4D case, the almost sure scattering in the radial case was proved in \cite{DLM19Adv,KMV19CPDE} and extended to the non-radial case in \cite{Spi21NLS} very recently. All the above 4D results consider the initial data in $H_x^s$ with $s>\frac12$, in which case the interaction Morawetz estimate holds for the original solution $u$, and more importantly, all the previous random data scattering results for both NLS and NLW are established in $H_x^s$ with $s>0$. In this paper, we introduce a new method to prove the almost sure global well-posedness as well as the scattering for non-radial NLS in both 3D and 4D cases, without imposing any regularity restriction on the initial data.

Taking the 3D case for example, the main difficulty comes from that the non-linear part of solution does not have any $L_x^p$-estimate with $p>6$. Therefore, we introduce a new strategy to control the energy increment, based on the $\De$-estimate of the linear flow and the modified interaction Morawetz estimate.

When considering the $L_x^2$ sub-critical data, it seems very difficult to establish the Strichartz estimates directly due to the lack of spherical symmetry. By introducing a new randomization, we are able to prove some $L_x^2$ super-critical estimates with good smoothing effect for the linear flow. 

Finally, we remarks that our method also works for $d=5$ and $d=6$ cases when the non-linear term is not necessarily algebraic, by slightly modifying the argument. We do not pursue this issue here, since the argument would be technically complex when invoking the fractional calculus. 

\subsection{Sketch of the proof.} The main ingredient of the proof is summarized as follows.

$\bullet$ \textit{High-low frequency decomposition.} 
The high-low decomposition was introduced by Bourgain \cite{Bou98IMRN}, and then first applied to the probabilistic setting by Colliander-Oh \cite{CO12duke}. 
Here we use the framework in our previous paper in \cite{SSW21cubicNLS}. More precisely,  in order to quantify the size of mass and energy, we first decompose the probability space $\Omega$ by setting 
\EQ{
\wt \Omega_M=\big\{&\omega\in\Om: \|f^\om \|_{H^s} + N_0^{s}\|P_{\loe N_0}f^\om \|_{L_x^2} \\
&+ N_0^{s-1}\|P_{\loe N_0}f^\om \|_{\dot H^1}+\|e^{it\Delta}f^\om\|_{Y(\R)}\loe M\norm{f}_{H^s}\big\},
}
where the $Y$ is some required space-time norm.

Then we  consider $\omega\in \wt\Omega_M$ for each $M$ separately, and make the high-low frequency decomposition as 
$$
v=e^{it\De}P_{\goe N_0}f^\om, \mbox{ and }\quad w=u-v.
$$ 
Then, for any $\omega \in\wt \Omega_M$, there exists a constant $C(M,\norm{f}_{H^s})>0$ such that 
\EQ{
M(0)\loe C(M,\norm{f}_{H^s})N_0^{-2s}\text{, and }E(0)\loe C(M,\norm{f}_{H^s})N_0^{2(1-s)}.
}

The framework here has two benefits:
\begin{itemize}
	\item[(1)] 
	$\wh v$ is supported on $\fbrk{|\xi|\gsm N_0}$.
	\item[(2)] 
	We can explicitly keep track of the mass and energy increment by the large dyadic parameter $N_0$ that is independent of $\om$.
\end{itemize} 

$\bullet$ \textit{Linear estimates with smoothing effect.} 

In this paper, we consider the nonlinear Schr\"odinger equation in non-radial case in $H^s$ for any $s\in \R$. 
We propose a new kind of randomization  based on the repartition of Wiener decomposition in the frequency space. According to the definition of $\bx_j$,  the following Bernstein-like estimate holds:
$$
\big\|\bx_jf\big\|_{L^\infty_x(\R^d)}\lesssim  \||\nabla|^{-\frac d2a} \bx_jf\|_{L^2_x(\R^d)}, 
 \quad\mbox{for each } j\in \N. 
$$
This implies that the smoothing effect only depends on the volume of support set of $\psi_Q$ and one may gain the arbitrary regularity ($|\nabla|^{\frac d2a}$) for non-radial function by choosing sufficiently large $a>0$. 
Due to this, we obtain that for any fixed $s\in\R$, any $f\in H^s(\R^d)$, by choosing  $a$  suitably large,  the following estimate is available:
\EQ{
\norm{\De e^{it\De}f^\om}_{L_t^2 L_x^\I}<+\I\text{, a. e. }\om\in\Om.
}

$\bullet$ \textit{Modified interaction Morawetz estimate.} The purpose here is to bound 
\EQ{
\normb{|\nabla|^{-\frac{d-3}{4}}w}_{L_{t,x}^4}
}
by the interaction Morawetz estimate.
The key point is that we modify the remainder avoiding the terms containing $w^{\frac{4}{d-2}}\nabla w$. More precisely, in the 3D case, the remainder includes the terms like
\EQ{
\int_{\R^3} vw^4\nabla \bar w \dx.
}
Since we only have $L_x^2$ estimate for $\nabla w$, we are forced to use $L_x^8$ for the remaining $w$. However, we only have $L_x^q$-estimate with $q\loe 6$ for $w$. This brings the difficulty to obtain the desired estimate. To overcome this difficulty, the main  technique here is to transfer the derivative from $w$ to $v$, and reduce it to the easy term 
\EQ{
\int_{\R^3} |w|^4w\nabla \bar v \dx.
}
Note that it can not be obtained by integration-by-parts directly, instead it follows by using the structure of nonlinearity. 

This kind of estimate may also be of independent interest. 

$\bullet$ \textit{Mass increment.} Note that when $s<0$, $u_0\notin L_t^\I L_x^2$, which is much different from most of the previous papers.  
Hence, we need the following almost conserved $L^2$-estimate firstly:
$$
\|w\|_{L^\infty_tL_x^2}\lesssim \|w_0\|_{L_x^2}.
$$
One may note that 
$$
\|w(t)\|_{L_x^2}^2=\|w_0\|_{L_x^2}^2+\int_I\int_{\R^d} O(w^{\frac{d+2}{d-2}}v)\,dxdt.
$$
The cubic growth (when $d=3$) of the increment make the obstruction to close the estimate by Gronwall's inequality. To cover the additional increment, we make use of the high-low frequency decomposition above, which  allows us to get extra regularity of $v$ which relies on the reciprocal of $\|w(t)\|_{H^1}$.

$\bullet$ \textit{Energy increment.} To prove the almost energy conservation law, which gives the bound of $H^1$, 
the main task is to control the energy increment
\EQ{
\int_I\int_{\R^d} \De v |w|^{\frac{4}{d-2}}\wb w\dx\dt.
}
Roughly speaking, based on the above estimates, the increment can be bounded by
\EQ{
\norm{\De v}_{L_t^2 L_x^\I} \normb{|\nabla|^{-\frac{d-3}{4}}w}_{L_{t,x}^4}^2 \normb{|\nabla|^{\frac{d-3}{2}}\brkb{w^{\frac{6-d}{d-2}}}}_{L_t^\I L_x^{2}}.
}
In 3D case,
\EQ{
\normb{|\nabla|^{\frac{d-3}{2}}\brkb{w^{\frac{6-d}{d-2}}}}_{L_t^\I L_x^{2}}\lsm \norm{w}_{L_t^\I L_x^6}^3,
}
and in 4D case,
\EQ{
\normb{|\nabla|^{\frac{d-3}{2}}\brkb{w^{\frac{6-d}{d-2}}}}_{L_t^\I L_x^{2}}\lsm \normb{|\nabla|^{\frac12}w}_{L_t^\I L_x^2}.
}
Again, we can cancel the additional energy increment using the $\norm{\De v}_{L_t^2 L_x^\I}$ estimate. This close the bootstrap procedure.

$\bullet$ \textit{Stability.} We also need to adopt the perturbation theory to approximate the perturbed equation \eqref{eq:nls-w}. This idea of iterating the perturbation theory was first employed by \cite{TVZ07CPDE}, and then was extended in the probabilistic setting by \cite{BOP15TranB} and \cite{KMV19CPDE}. This method was widely used in the almost sure scattering theory. In this paper, we are mainly inspired by the recent approach as in \cite{KMV19CPDE}, which is tailored to the equation \eqref{eq:nls-w}.

\subsection{Organization of the paper}
In Section \ref{sec:pre}, we give some notation and useful results. In Section \ref{sec:str}, we prove the almost sure space-time estimates for the linear solution. Then, we prove the 3D case of Theorem \ref{thm:global} in Section \ref{sec:gwp-3d}, and the 4D case in Section \ref{sec:gwp-4d}.

\vskip 1.5cm

\section{Preliminary}\label{sec:pre}

\vskip .5cm
\subsection{Notation}
For any $a\in\R$, $a\pm:=a\pm\epsilon$ for arbitrary small $\epsilon>0$. For any $z\in\C$, we define $\re z$ and $\im z$ as the real and imaginary part of $z$, respectively. For any set $A$, we denote $\sharp A$ as the cardinal number of $A$.

Let $C>0$ denote some constant, and write $C(a)>0$ for some constant depending on coefficient $a$. If $f\loe C g$, we write $f\lsm g$. If $f\loe C g$ and $g\loe C f$, we write $f\sim g$. Suppose further that $C=C(a)$ depends on $a$, then we write $f\lsm_a g$ and $f\sim_a g$, respectively. If $f\loe 2^{-5}g$, we denote $f\ll g$ or $g\gg f$.

Moreover, we write ``a.e. $\om\in\Om$'' to mean ``almost every $\om\in\Om$''.

We use $\wh f$ or $\F f$ to denote the Fourier transform of $f$:
\EQ{
\wh f(\xi)=\F f(\xi):= \int_{\R^d} e^{-ix\cdot\xi}f(x)\rm dx.
}
We also define
\EQ{
\F^{-1} g(x):= \rev{(2\pi)^d}\int_{\R^d} e^{ix\cdot\xi}g(\xi)\rm d\xi.
}
Using the Fourier transform, we can define the fractional derivative $\abs{\nabla} := \F^{-1}|\xi|\F $ and $\abs{\nabla}^s:=\F^{-1}|\xi|^s\F $.  


We also need the usual inhomogeneous Littlewood-Paley decomposition for the dyadic number. Take a cut-off function $\phi\in C_{0}^{\infty}(0,\infty)$ such that $\phi(r)=1$ if $r\loe1$ and $\phi(r)=0$ if $r>2$. 
For dyadic $N\in 2^\N$, when $N\goe 1$, let $\phi_{\loe N}(r) = \phi(N^{-1}r)$ and $\phi_N(r) =\phi_{\loe N}(r)-\phi_{\loe N/2}(r)$. We define the Littlewood-Paley dyadic operator $$f_{\loe N}=P_{\loe N} f := \mathcal{F}^{-1}\brko{ \phi_{\loe N}(|\xi|) \hat{f}(\xi)},$$ and $$f_{N}= P_N f := \mathcal{F}^{-1}\brko{ \phi_N(|\xi|) \hat{f}(\xi)}.$$ We also define that $f_{\goe N}=P_{\goe N} f := f- P_{\loe N} f$, $f_{\ll N}=P_{\ll N} f$, $f_{\gsm N}:=P_{\gsm N}f$, $f_{\lsm N}:=P_{\lsm N}f$, and $f_{\sim N}=P_{\sim N}f$.

Let $\Sch(\R^d)$ be the Schwartz space,  $\Sch'(\R^d)$ be the tempered distribution space, and $C_0^\I(\R^d)$ be the space of all the smooth compact-supported functions.

Given $1\loe p \loe \I$, $L^p(\R^d)$ denotes the usual Lebesgue space. We define the Sobolev space
\EQ{
	\dot W^{s,p}(\R^d) := \fbrkb{f\in\Sch'(\R^d): \norm{f}_{\dot{W}^{s,p}(\R^d)}:= \norm{\abs{\nabla}^{s}f}_{L^p(\R^d)}<+\I}.
}
We denote that $\dot{H}^s(\R^d):=\dot{W}^{s,2}(\R^d)$. The inhomogeneous spaces are defined by 
\EQ{
W^{s,p}(\R^d)=\dot W^{s,p} \cap L^p(\R^d)\text{, and }H^{s}(\R^d)=\dot H^{s} \cap L^2(\R^d).
}
We often use the abbreviations $H^s=H^s(\R^d)$ and $L^p=L^p(\R^d)$. We also define $\jb{\cdot,\cdot}$ as real $L^2$ inner product:
\EQ{
	\jb{f,g} = \re\int f(x)\wb{g}(x)\dx.
}

For any $1\loe p <\I$, define $l_N^p=l_{N\in 2^\N}^p$ by its norm
\EQ{
\norm{c_N}_{l_{N\in 2^\N}^p}^p:=\sum_{N\in 2^\N}|c_N|^p.
}
The space $l_j^p=l_{j\in\N}^p$ is defined in a similar way. 

We then define the mixed norms: for $1\loe q< \I$, $1\loe r\loe \I$, and the function $u(t,x)$, we define
\EQ{
\norm{u}_{L_t^q L_x^r(\R\times \R^d)}^q:= \int_{\R}\norm{u(t,\cdot)}_{L_x^r}^q\dt,
}
and for the function $u_N(x)$, we define
\EQ{
\norm{u_N}_{l_N^q L_x^r(2^\N\times \R^d)}^q:= \sum_{N}\norm{u_N(\cdot)}_{L_x^r}^q.
}
The $q=\I$ case can be defined similarly.

For any $0\loe\gamma\loe1$, we call that the exponent pair $(q,r)\in\R^2$ is $\dot H^\ga$-$admissible$, if $\frac{2}{q}+\frac{d}{r}=\half d-\ga$, $2\loe q\loe\I$, $2\loe r\loe\I$, and $(q,r,d)\ne(2,\I,2)$. If $\ga=0$, we say that $(q,r)$ is $L^2$-$admissible$.

\subsection{Useful lemmas}
In this subsection, we gather some useful results.
\begin{lem}[Hardy's inequality]\label{lem:hardy}
	For $0<s<d/2$, we have that
	\EQ{
		\normb{|x|^{-s}u}_{L_x^2(\R^d)}\lsm \norm{ u}_{H_x^s(\R^d)}.
	}
\end{lem}

\begin{lem}[Gagliardo-Nirenberg inequality]\label{lem:GN}
	Let $d\goe 1$, $0<\th<1$, $0<s_1<s_2$ and $1<p_1,p_2,p_3\loe\I$. Suppose that $\rev{p_1}= \frac\th{p_2}+\frac{1-\th}{p_3}$ and $s_1=\th s_2$. Then, we have
	\EQn{\label{eq:GN}
		\norm{\abs{\nabla}^{s_1}u}_{L_x^{p_1}(\R^d)} \lsm \norm{\abs{\nabla}^{s_2}u}_{L_x^{p_2}(\R^d)}^\th \norm{u}_{L_x^{p_3}(\R^d)}^{1-\th}.
	}
\end{lem}

Then by Lemma \ref{lem:GN}, we can easily obtain
\begin{lem}\label{lem:GN-2}
	Let $d\goe 1$, $0<\th<1$, $0<s_1<s_2$ and $1<p_1,p_2,p_3\loe\I$. Suppose that $\rev{p_1}= \frac\th{p_2}+\frac{1-\th}{p_3}$ and $s_1\loe\th s_2$. Then, we have
	\EQn{\label{eq:GN-2}
		\norm{\jb{\nabla}^{s_1}u}_{L_x^{p_1}(\R^d)} \lsm \norm{\jb{\nabla}^{s_2}u}_{L_x^{p_2}(\R^d)}^\th \norm{u}_{L_x^{p_3}(\R^d)}^{1-\th}.
	}
\end{lem}

\begin{lem}[Strichartz estimate, \cite{KT98AJM}]\label{lem:strichartz}
	Let $I\subset \R$. Suppose that $(q,r)$ and $(\wt{q},\wt{r})$ are $L_x^2$-admissible.
	Then,
	\EQn{\label{eq:strichartz-1}
		\norm{ e^{it\De}\ph}_{L_t^qL_x^r(\R\times\R^d)} \lsm \norm{\ph}_{L_x^2},
	}
	and
	\EQn{\label{eq:strichartz-2}
		\normb{\int_0^t e^{i(t-s)\De} F(s)\ds}_{L_t^qL_x^r(\R\times \R^d)} \lsm \norm{F}_{L_t^{\wt{q}'} L_x^{\wt{r}'}(\R\times\R^d)}.
	}
\end{lem}

\begin{lem}[Littlewood-Paley estimates]\label{lem:littlewood-paley}
	Let $1<p<\I$ and $f\in L_x^p(\R^d)$. Then, we have
	\EQ{
		\norm{f_N}_{L_x^p l_{N\in2^\N}^2}\sim_p \norm{f}_{L_x^p}.
	}
\end{lem}
Next, we give some properties of $\bx_j$.
\begin{lem}[Orthogonality]\label{lem:orthogonality}
	Let $f\in L_x^2(\R^d)$. Then, we have
	\EQ{
		\norm{\bx_jf}_{L_x^2 l_{j\in\N}^2}\sim \norm{f}_{L_x^2}.
	}
\end{lem}
\begin{proof}
Since $\Q$ is a set of essentially disjoint cubes, for any fixed $j\in \N$, let
$$
\B_j := \fbrk{j'\in\N:\supp \psi_j \cap \supp \psi_{j'}\ne\emptyset},
$$
then $\sharp \B_j\lsm 1$. We also have
\EQ{
j'\in \B_j \Leftrightarrow j\in\B_{j'}.
}
Then by Plancherel's identity,
\EQ{
\norm{f}_{L_x^2(\R^d)}^2= & \normb{\sum_{j\in \N}\bx_jf}_{L_x^2(\R^d)}^2 \\
= & \int_{\R^d} |\sum_{j\in \N} \psi_j(\xi)\wh f(\xi)|^2 \mathrm d\xi \\
= & \sum_{j,j'\in \N:j'\in \B_j} \int_{\R^d} \psi_j(\xi) \psi_{j'}(\xi) | \wh f(\xi)|^2 \mathrm d\xi.
}
Hence, on one hand, 
\EQ{
\norm{\bx_jf}_{L_x^2 l_{j\in \N}^2} = & \sum_{j\in \N} \int_{\R^d} | \psi_j(\xi) \wh f(\xi)|^2 \mathrm d\xi \\
\loe & \sum_{j,j'\in \N:j'\in \B_j} \int_{\R^d} \psi_j(\xi) \psi_{j'}(\xi) | \wh f(\xi)|^2 \mathrm d\xi = \norm{f}_{L_x^2(\R^d)}^2.
}
On the other hand, by Cauchy-Schwartz's inequality,
\EQ{
\norm{f}_{L_x^2(\R^d)}^2= & \sum_{j,j'\in \N:j'\in \B_j} \int_{\R^d} \psi_j(\xi) \psi_{j'}(\xi) | \wh f(\xi)|^2 \mathrm d\xi \\
\lsm & \sum_{j,j'\in \N:j'\in \B_j} \brkb{\int_{\R^d} |\psi_j(\xi) \wh f(\xi)|^2 \mathrm d\xi +  \int_{\R^d} |\psi_{j'}(\xi) \wh f(\xi)|^2 \mathrm d\xi} \\
\lsm & \sum_{j,j'\in \N:j'\in \B_j} \int_{\R^d} |\psi_j(\xi) \wh f(\xi)|^2 \mathrm d\xi \lsm \norm{\bx_jf}_{L_x^2 l_{j\in\N}^2}.
}
This finishes the proof.
\end{proof}

We then prove a Bernstein-like estimate for $\bx_j$. This provides the required smoothing effect, as long as the scale of cube $N^{-a}$ is suitably small.
\begin{lem}[$L^q$-$L^p$ estimate]\label{lem:bernstein}
Let $a>0$ and $2\loe p\loe q \loe\I$. Given any $j\in\N$, then
\EQ{
\norm{\bx_j f}_{L_x^q(\R^d)}\lsm \normb{\jb{\nabla}^{-a(\frac dp-\frac dq)}\bx_j  f}_{L_x^p(\R^d)}.
}
\end{lem}
\begin{proof}
By the property of $\Q$, given any $j\in\N$, there exists $N\in2^\N$ such that for all $\xi\in \supp\psi_j$,
\EQ{
	\jb{|\xi|}\sim N\text{, and }|\supp \psi_j| \sim N^{-da}.
}

Let $r\goe 2$ and $0\loe \th\loe 1$ such that $\frac 1r=\frac1{q'}-\frac12$ and $\frac1p=\frac{\th}{2} + \frac{1-\th}{q}$. Now, by Hausdorff-Young's and H\"older's inequalities,
\EQ{
\norm{\bx_j f}_{L_x^q(\R^d)} \loe & \normb{\psi_j \wh f}_{L_\xi^{q'}(\R^d)} \\
\lsm & \normb{\mathbbm{1}_{\xi\in 2Q_j}}_{L_\xi^r(\R^d)} \normb{\psi_j \wh f}_{L_x^2(\R^d)} \\
\lsm & N^{-a\cdot\frac dr} \norm{\bx_j f}_{L_x^2(\R^d)}.
}
Note that we have
\EQ{
\frac{1}{r}=\frac12-\frac1q\text{, and }\frac1p=\frac1q+\frac\th r.
}
Then, interpolating with the trivial estimate $\norm{\bx_j f}_{L_x^q(\R^d)}\lsm \norm{\bx_j f}_{L_x^q(\R^d)}$, for $2\loe p\loe q$,
\EQ{
\norm{\bx_j f}_{L_x^q(\R^d)} \lsm & N^{-a\cdot\frac dr\cdot \th} \norm{\bx_j f}_{L_x^p(\R^d)} \\
\lsm & N^{-a(\frac dp-\frac dq)} \norm{\bx_j f}_{L_x^p(\R^d)} \\
\lsm & \normb{\jb{\nabla}^{-a(\frac dp-\frac dq)} \bx_j f}_{L_x^p(\R^d)}.
}
This completes the proof of this lemma.
\end{proof}

We remark that this lemma is different from the usual Bernstein' inequality for the Littlewood-Paley projection operator $P_N$: for $1\loe p\loe q\loe \I$ and $N\in 2^{\N}$,
\EQ{
\norm{P_N f}_{L_x^q(\R^d)}\lsm N^{\frac dp-\frac dq}\norm{P_N f}_{L_x^p(\R^d)}.
}
In fact, the cut-off $\phi_N$ can be generated by rescaling, so this estimate follows by Young's inequality. However, $\psi_j$ does not have that property. Therefore, we need to use the Hausdoff-Young's inequality. That's the reason why we require $p\goe 2$.

\subsection{Probabilistic theory}
We recall the large deviation estimate, which holds for the random variable sequence $\fbrk{\re g_k,\im g_k}$ in the Definition \ref{defn:randomization}.
\begin{lem}[Large deviation estimate, \cite{BT08inventI}]\label{lem:large-deviation}
Let $(\Om, \A, \PP)$ be a probability space. Let $\fbrk{g_n}_{n\in\N^+}$ be a sequence of real-valued, independent, zero-mean random variables with associated distributions $\fbrk{\mu_n}_{n\in\N^+}$ on $\Om$. Suppose $\fbrk{\mu_n}_{n\in\N^+}$ satisfies that there exists $c>0$ such that for all $\ga\in\R$ and $n\in\N^+$
\EQ{
\absb{\int_{\R}e^{\ga x}\mathrm d \mu_n(x)}\loe e^{c\ga^2},
}
then there exists $\al>0$ such that for any $\la>0$ and any complex-valued sequence $\fbrk{c_n}_{n\in\N^+}\in l_n^2$, we have
\EQ{
\PP\brkb{\fbrkb{\om:\absb{\sum_{n=1}^\I c_n g_n(\om)}>\la}}\loe 2\exp\fbrkb{-\al\la\norm{c_n}_{l_n^2}^{-2}}.
}
Furthermore, there exists $C>0$ such that for any $2\loe p<\I$ and complex-valued sequence $\fbrk{c_n}_{n\in\N^+}\in l_n^2$, we have
\EQn{\label{eq:large-deviation}
\normb{\sum_{n=1}^\I c_n g_n(\om)}_{L_\om^p(\Om)} \loe C\sqrt{p} \norm{c_n}_{l_n^2}.
}
\end{lem}
The following lemma can be proved by the method in \cite{Tzv10gibbs}, see also \cite{DLM20AJM,DLM19Adv}.
\begin{lem}\label{lem:probability-estimate}
Let $F$ be a real-valued measurable function on a probability space $(\Om, \A, \PP)$. Suppose that there exists $C_0>0$, $K>0$ and $p_0\goe 1$ such that for any $p\goe p_0$, we have
\EQ{
\norm{F}_{L_\om^p(\Om)} \loe \sqrt{p} C_0 K.
}
Then, there exist $c>0$ and $C_1>0$, depending on $C_0$ and $p_0$ but independent of $K$, such that for any $\la>0$,
\EQ{
\PP\brkb{\fbrkb{\om\in\Om:|F(\om)|>\la}} \loe C_1 e^{-c\la^2K^{-2}}.
}
Particularly, we have
\EQ{
\PP\brkb{\fbrkb{\om\in\Om:|F(\om)|<\I}}=1.
}
\end{lem}

\vskip 1.5cm

\section{Almost sure Strichartz estimates}\label{sec:str}

\vskip .5cm

\subsection{Strichartz estimates}
\begin{lem}\label{lem:almostsure-strichartz}
Let $d\goe 1$ and $f\in L_x^2(\R^d)$. Suppose that the randomization $f^\om$ is defined in Definition \ref{defn:randomization}. Then, we have the following estimates:
\enu{
\item 
Given any $2\loe q, r<\I$ with $\frac2q+\frac dr\loe \frac d2$ and let $2\loe r_0<\I$ such that $(q,r_0)$ is $L_x^2$-admissible. Then for any $p\goe \max\fbrk{q,r}$ and $0\loe s\loe a(\frac{d}{r_0}-\frac{d}{r})$,
\EQn{\label{eq:bound-v-pro-ltqxr}
\norm{\jb{\nabla}^se^{it\De}f^\om}_{L_\om^pL_t^q L_x^r(\Om\times\R\times\R^d)}\lsm\sqrt{p}\norm{f}_{L_x^2(\R^d)}.
}
\item 
For any $p\goe2$,
\EQn{\label{eq:bound-v-pro-ltinftyx2}
	\norm{e^{it\De}f^\om}_{L_\om^pL_t^\I L_x^2(\Om\times\R\times\R^d)}\lsm\sqrt{p}\norm{f}_{L_x^2(\R^d)}.
}
\item 
Given any $2\loe q<\I$ and let $2\loe r_0<\I$ such that $(q,r_0)$ is $L_x^2$-admissible. For any $0\loe s<\frac{d}{r_0}\cdot a$, there exists $p_0\goe 2$ such that for any $p\goe p_0$,
\EQn{\label{eq:bound-v-pro-ltqxinfty}
\norm{\jb{\nabla}^{s}e^{it\De}f^\om}_{L_\om^pL_t^q L_x^\I(\Om\times\R\times\R^d)}\lsm\sqrt{p}\norm{f}_{L_x^2(\R^d)}.
}
\item 
Given any $2< r\loe\I$. For any $0\loe s<(\frac d2-\frac{d}{r})\cdot a$, there exists $p_0\goe 2$ such that for any $p\goe p_0$,
\EQn{\label{eq:bound-v-pro-ltinftyxr}
\norm{\jb{\nabla}^{s}e^{it\De}f^\om}_{L_\om^pL_t^\I L_x^r(\Om\times\R\times\R^d)}\lsm\sqrt{p}\norm{f}_{L_x^2(\R^d)}.
}
} 
\end{lem}
\begin{proof}
In the proof of this lemma, we restrict the variables on $\om\in\Om$, $t\in\R$, $x\in\R^d$, and $j\in\N$.

We first prove \eqref{eq:bound-v-pro-ltqxr}. By Minkowski's inequality and Lemma \ref{lem:large-deviation}, we have
\EQn{\label{esti:bound-v-pro-ltqxr-1}
\norm{\jb{\nabla}^se^{it\De}f^\om}_{L_\om^pL_t^q L_x^r}\lsm & \norm{\jb{\nabla}^se^{it\De}f^\om}_{L_t^q L_x^rL_\om^p} \\
\lsm & \sqrt{p} \norm{\jb{\nabla}^se^{it\De}\bx_jf}_{L_t^q L_x^rl_j^2} \\
\lsm & \sqrt{p} \norm{\jb{\nabla}^se^{it\De}\bx_jf}_{l_j^2L_t^q L_x^r}.
}
Now, let $2\loe r_0\loe r$ such that $(q,r_0)$ is $L_x^2$-admissible. 
Then, for any $k\in\N^+$, by Lemma \ref{lem:bernstein}, we have
\EQn{\label{esti:bound-v-pro-ltqxr-2}
\norm{\jb{\nabla}^se^{it\De}\bx_jf}_{L_t^q L_x^r}
\lsm & \normb{\jb{\nabla}^{s-a(\frac{d}{r_0}-\frac{d}{r})}e^{it\De}\bx_jf}_{L_t^q L_x^{r_0}}\\
\lsm & \normb{e^{it\De}\bx_jf}_{L_t^q L_x^{r_0}}.
}
Then, by \eqref{esti:bound-v-pro-ltqxr-1}, \eqref{esti:bound-v-pro-ltqxr-2}, Lemmas \ref{lem:strichartz}, and \ref{lem:orthogonality}, we have
\EQ{
\norm{\jb{\nabla}^se^{it\De}f^\om}_{L_\om^pL_t^q L_x^r}\lsm & \sqrt{p} \normb{e^{it\De}\bx_jf}_{L_t^q L_x^{r_0}} \\
\lsm & \sqrt{p} \norm{ \bx_jf}_{l_j^2 L_x^2}\lsm \sqrt{p} \norm{f}_{L_x^2}.
}
This gives \eqref{eq:bound-v-pro-ltqxr}.

Next, we prove \eqref{eq:bound-v-pro-ltinftyx2}. By Plancherel's identity, 
\EQn{\label{esti:bound-v-pro-ltinftyx2-1}
	\norm{\jb{\nabla}^{s}e^{it\De}f^\om}_{L_\om^pL_t^\I L_x^2}\lsm \norm{\jb{\nabla}^{s}f^\om}_{L_\om^p L_x^2}\lsm \norm{\jb{\nabla}^{s}f^\om}_{L_\om^p L_x^2}.
}
Then for $p\goe 2$, by Minkowski's inequality, Lemmas \ref{lem:large-deviation}, and \ref{lem:orthogonality},
\EQn{\label{esti:bound-v-pro-ltinftyx2-2}
	\norm{\jb{\nabla}^{s}f^\om}_{L_\om^p L_x^2}\lsm & \norm{\jb{\nabla}^{s}f^\om}_{L_x^2 L_\om^p} \\
	\lsm & \sqrt{p} \norm{\jb{\nabla}^{s}\bx_j f}_{L_x^2 l_j^2}\lsm \sqrt{p}\norm{f}_{H_x^s}.
}
Then, \eqref{esti:bound-v-pro-ltinftyx2-1} and \eqref{esti:bound-v-pro-ltinftyx2-2} imply \eqref{eq:bound-v-pro-ltinftyx2}.

We then prove \eqref{eq:bound-v-pro-ltqxinfty}. Let $0<\ep\loe \frac{1}{a+2}(\frac{d}{r_0}\cdot a-s)$ such that
\EQ{
s+2\ep-a(\frac{d}{r_0}-\ep)\loe 0.
}
Using the Sobolev's embedding $ W_x^{2\ep,\frac d\ep} \hra L_x^\I$ in $x$, we have
\EQn{\label{esti:bound-v-pro-ltqxinfty-1}
\norm{\jb{\nabla}^{s}e^{it\De}f^\om}_{L_\om^pL_t^q L_x^\I}\lsm \norm{\jb{\nabla}^{s+2\ep}e^{it\De}f^\om}_{L_\om^pL_t^q L_x^{\frac d\ep}}.
}
Let $p_0=\max\fbrk{q,\frac d\ep}$ and $2\loe r_0\loe\frac d\ep$ such that $(q,r_0)$ is $L_x^2$-admissible. Then, similar as above, by Minkowski's inequality, Lemmas \ref{lem:strichartz}, \ref{lem:bernstein}, \ref{lem:orthogonality}, and \ref{lem:large-deviation}, for any $p\goe p_0$, we have
\EQn{\label{esti:bound-v-pro-ltqxinfty-2}
\norm{\jb{\nabla}^{s+2\ep}e^{it\De}f^\om}_{L_\om^pL_t^q L_x^{\frac d\ep}} \lsm & \norm{\jb{\nabla}^{s +2\ep}e^{it\De}f^\om}_{L_t^q L_x^{\frac d\ep}L_\om^p} \\
\lsm & \sqrt{p} \norm{\jb{\nabla}^{s+2\ep}e^{it\De}\bx_jf}_{L_t^q L_x^{\frac d\ep}l_j^2} \\
\lsm & \sqrt{p} \norm{\jb{\nabla}^{s+2\ep}e^{it\De}\bx_jf}_{l_j^2L_t^q L_x^{\frac d\ep}} \\
\lsm & \sqrt{p} \norm{\jb{\nabla}^{s+2\ep -a(\frac{d}{r_0}-\ep)}e^{it\De}\bx_jf}_{l_j^2L_t^q L_x^{r_0}} \\
\lsm & \sqrt{p} \norm{\bx_jf}_{l_j^2 L_x^2}\lsm \sqrt{p} \norm{f}_{L_x^2}.
}
By \eqref{esti:bound-v-pro-ltqxinfty-1} and \eqref{esti:bound-v-pro-ltqxinfty-2}, we have that \eqref{eq:bound-v-pro-ltqxinfty} holds.

Finally, we prove \eqref{eq:bound-v-pro-ltinftyxr}. We only consider the $r=\I$ case. In fact, when $r<\I$, we can prove it using interpolation between \eqref{eq:bound-v-pro-ltinftyx2} and the $r=\I$ case. Let $0\loe s<\frac d2\cdot a$ and some sufficiently small $0<\ep\loe \frac{1}{6+3a}(\frac d 2 a-s)$ such that
\EQ{
s+6\ep-a(\frac{d}{2}-3\ep)\loe 0.
}
Using Minkowski's inequality, the Sobolev's embeddings $W_t^{2\ep,\frac1\ep}\hra L_t^\I$ in $t$, and $ W_x^{2\ep,\frac d\ep} \hra L_x^\I$ in $x$, we have
\EQn{\label{esti:bound-v-pro-ltinftyxr-1}
\norm{\jb{\nabla}^{s}e^{it\De}f^\om}_{L_\om^pL_t^\I L_x^\I}
\lsm & 	\norm{\jb{\nabla}^{s+2\ep}e^{it\De}f^\om}_{L_\om^pL_t^\I L_x^{\frac{3}{\ep}}} \\
\lsm & \norm{\jb{\nabla}^{s+2\ep}\jb{\pd_t}^{2\ep}e^{it\De}f^\om}_{L_\om^p L_x^{\frac{3}{\ep}} L_t^{\frac1\ep}}\\
\lsm & \norm{\jb{\nabla}^{s+6\ep}e^{it\De}f^\om}_{L_\om^p L_x^{\frac{3}{\ep}} L_t^{\frac1\ep}}\\
\lsm & \norm{\jb{\nabla}^{s+6\ep}e^{it\De}f^\om}_{L_\om^p  L_t^{\frac1\ep} L_x^{\frac{3}{\ep}}}.
}
Let $p_0=\frac3\ep$ and $r_0=\frac{2d}{d-4\ep}$ such that $2< r_0 \loe \frac3\ep$ and $(\frac1\ep,r_0)$ is $L_x^2$-admissible. By the choice of $\ep$ and $r_0$, we have
\EQ{
s+6\ep-a(\frac{d}{r_0}-\ep)\loe 0.
}
Then, similar as above, for any $p\goe p_0$, 
\EQn{\label{esti:bound-v-pro-ltinftyxr-2}
	\norm{\jb{\nabla}^{s+6\ep}e^{it\De}f^\om}_{L_\om^pL_t^{\frac1\ep} L_x^{\frac3\ep}} \lsm & \norm{\jb{\nabla}^{s+6\ep}e^{it\De}f^\om}_{L_t^{\frac1\ep} L_x^{\frac3\ep} L_\om^p} \\
	\lsm & \sqrt{p} \norm{\jb{\nabla}^{s+6\ep}e^{it\De}\bx_jf}_{L_t^{\frac1\ep} L_x^{\frac3\ep} l_j^2} \\
	\lsm & \sqrt{p} \norm{\jb{\nabla}^{s+6\ep}e^{it\De}\bx_jf}_{l_j^2 L_t^{\frac1\ep} L_x^{\frac3\ep}} \\
	\lsm & \sqrt{p} \norm{\jb{\nabla}^{s+6\ep-a(\frac{d}{r_0}-\ep)}e^{it\De}\bx_jf}_{l_j^2 L_t^{\frac1\ep} L_x^{r_0}} \\
	\lsm & \sqrt{p} \norm{\bx_jf}_{l_j^2 L_x^2} \lsm \sqrt{p} \norm{f}_{L_x^2}.
}
By \eqref{esti:bound-v-pro-ltinftyxr-1} and \eqref{esti:bound-v-pro-ltinftyxr-2}, we have that \eqref{eq:bound-v-pro-ltinftyxr} holds.
\end{proof}

We make a few remarks about this randomization. We prove the Strichartz estimates with smoothing effect, namely some gain of derivative quantified by $|\nabla|^{Ca}$ with some $C>0$. Since we do not assume the radial condition, $L_t^q L_x^r$-estimates with $\frac{2}{q} + \frac dr>\frac d2$ are unavailable. Moreover, when $f\in H_x^s$ with $s<0$, we are also lack of the $L^2$-critical estimates in the form of 
\EQ{
\norm{e^{it\De}f^\om}_{L_t^q L_x^r}
}
with $L^2$-admissible $(q,r)$. However, we can expect the $L_t^q L_x^r$ estimates hold for $L^2$ super-critical scaling, by covering the $s$-order derivative by the smoothing effect. 

Next, we gather all the space-time norms that will be used below, splitting into two cases.
\subsection{Linear estimates in 3D case}
Define the $Y(I)$ space by its norm in 3D case,
\EQn{\label{defn:ys-norm-3d}
\norm{v}_{Y(I)}:=& \normb{\jb{\nabla}^{s+\frac12a-}v}_{L_t^2 L_x^\I(I\times \R^3)} + \norm{v}_{L_{t,x}^8(I\times \R^3)} + \norm{v}_{L_{t,x}^4(I\times \R^3)} \\
&+ \norm{v}_{L_t^8 L_x^{12}(I\times \R^3)}.
}
We also define the $Z$-norm by
\EQn{\label{defn:z-norm}
\norm{v}_{Z(I)}:=  \norm{v}_{L_t^\I H_x^s(I\times\R^3)} 
+ \normb{\jb{\nabla}^{s+\frac{3}{2} a-}v}_{L_t^\I L_x^{\I}(I\times\R^3)}.
}

\begin{cor}\label{cor:Y-norm-Z-norm}
Let $s\in\R$ and $f\in H_x^s(\R^3)$. Suppose that the randomization $f^\om$ is defined in Definition \ref{defn:randomization}. Then, there exist constants $C,c>0$ such that for any $\la$,
\EQn{\label{eq:Y-norm-Z-norm-1}
	\PP\brkb{\fbrkb{\om\in \Om:\norm{e^{it\De}f^\om}_{Y(\R)} + \norm{e^{it\De}f^\om}_{Z(\R)}>\la}}\loe C\exp\fbrkb{-c\la^2\norm{f}_{H_x^s(\R^3)}^{-2}},
}
and we also have
\EQn{\label{eq:Y-norm-Z-norm-2}
\norm{e^{it\De}f^\om}_{Y(\R)} + \norm{e^{it\De}f^\om}_{Z(\R)} <+\I,\quad \mbox{a.e.} \>\>\om\in\Om.
}
\end{cor}
\begin{proof}
In the proof of this corollary, we restrict the variables on $\om\in\Om$, $t\in\R$, $x\in\R^3$. Fix a sufficiently large $p_0>0$, and let $p\goe p_0$. First, by \eqref{eq:bound-v-pro-ltqxinfty}, 
\EQ{
\normb{\jb{\nabla}^{s+\frac12a-}e^{it\De}f^\om}_{L_\om^pL_t^2 L_x^\I}\lsm\sqrt{p}\norm{f}_{H_x^s},
}
and by \eqref{eq:bound-v-pro-ltqxr},
\EQ{
\normb{\jb{\nabla}^{s+\frac78a}e^{it\De}f^\om}_{L_\om^pL_{t,x}^8}\lsm\sqrt{p}\norm{f}_{H_x^s},\\
\normb{\jb{\nabla}^{s+\frac14a}e^{it\De}f^\om}_{L_\om^pL_{t,x}^4}\lsm\sqrt{p}\norm{f}_{H_x^s},\\
\norm{\jb{\nabla}^{s+a}e^{it\De}f^\om}_{L_\om^pL_t^8 L_x^{12}}\lsm\sqrt{p}\norm{f}_{H_x^s}.
}
Since $s>\frac34-\frac14a$, we have
\EQ{
\normb{e^{it\De}f^\om}_{L_\om^pL_{t,x}^8} + \normb{e^{it\De}f^\om}_{L_\om^pL_{t,x}^4} + \normb{e^{it\De}f^\om}_{L_\om^pL_t^8 L_x^{12}}\lsm\sqrt{p}\norm{f}_{H_x^s}.
}
Therefore, we have
\EQ{
\normb{e^{it\De}f^\om}_{L_\om^p Y }\lsm\sqrt{p}\norm{f}_{H_x^s}.
}
By \eqref{eq:bound-v-pro-ltinftyx2} and \eqref{eq:bound-v-pro-ltinftyxr},
\EQ{
	\normb{e^{it\De}f^\om}_{L_\om^p Z }\lsm\sqrt{p}\norm{f}_{H_x^s}.
}
Then, by Lemma  \ref{lem:probability-estimate}, we have \eqref{eq:Y-norm-Z-norm-1} and \eqref{eq:Y-norm-Z-norm-2} hold.
\end{proof}

\subsection{Linear estimates in 4D case}
Define the $Y(I)$ space by its norm in 4D case,
\EQn{\label{defn:ys-norm-4d}
	\norm{v}_{Y(I)}:=&\normb{\jb{\nabla}^{s+a-}v}_{L_t^2 L_x^\I(I\times \R^4)} + \norm{v}_{L_t^4 L_x^8(I\times \R^4)} + \norm{v}_{L_t^6 L_x^3(I\times \R^4)} \\
	& + \normb{\jb{\nabla}^{-\frac14}v}_{L_{t,x}^4(I\times \R^4)}.
}
We also define the $Z$-norm by
\EQn{\label{defn:z-norm-4d}
	\norm{v}_{Z(I)}:=  \norm{v}_{L_t^\I H_x^s(\R\times\R^4)} 
	+ \normb{\jb{\nabla}^{s+2 a-}v}_{L_t^\I L_x^{\I}(\R\times\R^4)}.
}

\begin{cor}\label{cor:Y-norm-Z-norm-4d}
	Let $s\in\R$ and $f\in H_x^s(\R^4)$. Suppose that the randomization $f^\om$ is defined in Definition \ref{defn:randomization}. Then, there exist constants $C,c>0$ such that for any $\la$,
	\EQn{\label{eq:Y-norm-Z-norm-1-4d}
		\PP\brkb{\fbrkb{\om\in \Om:\norm{e^{it\De}f^\om}_{Y(\R)} + \norm{e^{it\De}f^\om}_{Z(\R)}>\la}}\loe C\exp\fbrkb{-c\la^2\norm{f}_{H_x^s(\R^4)}^{-2}},
	}
	and we also have
	\EQn{\label{eq:Y-norm-Z-norm-2-4d}
		\norm{e^{it\De}f^\om}_{Y(\R)} + \norm{e^{it\De}f^\om}_{Z(\R)} <+\I,\quad \mbox{a.e.} \>\>\om\in\Om.
	}
\end{cor}
\begin{proof}
	In the proof of this corollary, we restrict the variables on $\om\in\Om$, $t\in\R$, $x\in\R^4$. Fix a sufficiently large $p_0>0$, and let $p\goe p_0$. First, by \eqref{eq:bound-v-pro-ltqxinfty}, 
	\EQ{
		\normb{\jb{\nabla}^{s+a-}e^{it\De}f^\om}_{L_\om^pL_t^2 L_x^\I}\lsm\sqrt{p}\norm{f}_{H_x^s},
	}
	and by \eqref{eq:bound-v-pro-ltqxr},
	\EQ{
		\normb{\jb{\nabla}^{s+a}e^{it\De}f^\om}_{L_\om^pL_t^4 L_x^8}\lsm\sqrt{p}\norm{f}_{H_x^s},\\
		\normb{\jb{\nabla}^{s+\frac13a}e^{it\De}f^\om}_{L_\om^pL_t^6 L_x^3}\lsm\sqrt{p}\norm{f}_{H_x^s},\\
		\norm{\jb{\nabla}^{s+\frac12a}e^{it\De}f^\om}_{L_\om^pL_{t,x}^4}\lsm\sqrt{p}\norm{f}_{H_x^s}.
	}
	Since $s>\frac12-\frac12a$, we have
	\EQ{
		\normb{e^{it\De}f^\om}_{L_\om^pL_t^4 L_x^8} + \normb{e^{it\De}f^\om}_{L_\om^pL_t^6 L_x^3} + \normb{\jb{\nabla}^{-\frac14}e^{it\De}f^\om}_{L_\om^pL_{t,x}^4}\lsm\sqrt{p}\norm{f}_{H_x^s}.
	}
	Therefore, we have
	\EQ{
		\normb{e^{it\De}f^\om}_{L_\om^p Y }\lsm\sqrt{p}\norm{f}_{H_x^s}.
	}
	By \eqref{eq:bound-v-pro-ltinftyx2} and \eqref{eq:bound-v-pro-ltinftyxr},
	\EQ{
		\normb{e^{it\De}f^\om}_{L_\om^p Z }\lsm\sqrt{p}\norm{f}_{H_x^s}.
	}
	Then, by Lemma  \ref{lem:probability-estimate}, we have \eqref{eq:Y-norm-Z-norm-1-4d} and \eqref{eq:Y-norm-Z-norm-2-4d} hold.
\end{proof}

\vskip 1.5cm

\section{Global well-posedness and scattering in 3D case}\label{sec:gwp-3d}

\vskip .5cm

\subsection{Reduction to the deterministic problem}\label{sec:reduction-global}
Suppose that $u=v+w$ with $u_0=v_0+w_0$, $v=e^{it\De}v_0$, and $w$ satisfying
\EQn{
	\label{eq:nls-w}
	\left\{ \aligned
	&i\pd_t w + \De w = |u|^4 u, \\
	& w(0,x) = w_0(x).
	\endaligned
	\right.
}
Recall that
\EQ{
\norm{v}_{Y(I)}=& \normb{\jb{\nabla}^{s+\frac12a-}v}_{L_t^2 L_x^\I(I\times \R^3)} + \norm{v}_{L_{t,x}^8(I\times \R^3)} + \norm{v}_{L_{t,x}^4(I\times \R^3)} \\
&+ \norm{v}_{L_t^8 L_x^{12}(I\times \R^3)},
}
and
\EQ{
\norm{v}_{Z(I)}=\norm{v}_{L_t^\I H_x^s(\R\times\R^3)} 
+ \normb{\jb{\nabla}^{s+\frac{3}{2} a-}v}_{L_t^\I L_x^{\I}( \times\R\times\R^3)}.
}
We define the energy as
\EQn{\label{defn:energy-w}
	E(t):= \frac12\int_{\R^3}|\nabla w(t,x)|^2 \dx + \frac{1}{6}\int_{\R^3}|u(t,x)|^{6}\dx,
}
and the mass as
\EQn{\label{defn:mass-w}
	M(t):=\int_{\R^3}|w(t,x)|^2\dx.
}
Now, we reduce the proof of Theorem \ref{thm:global} in 3D case when $s<0$ to the following deterministic problem:
\begin{prop}\label{prop:global-derterministic}
Let $a\in\N$, $a> 10$, $\frac34-\frac14a<s<0$, and $A>0$. Then, there exists $N_0=N_0(A)\gg1$ such that the following properties hold.  Let $u_0\in H_x^s(\R^3)$, $v_0$ satisfy that $\supp\wh v_0\subset \fbrk{\xi\in\R^3:|\xi|\goe \frac12 N_0}$,  and $w_0=u_0-v_0$. Moreover, let $v=e^{it\Delta}v_0$ and $w=u-v$. Suppose that  $v\in Y\cap Z(\R)$, $w_0\in H^1(\R^3)$ such that
\EQ{
	\norm{u_0}_{H_x^s(\R^3)}+\norm{v}_{Y\cap Z(\R)}\loe A\text{, } M(0)\loe A N_0^{-2s}\text{, and }E(0) \loe A N_0^{2(1-s)}.
}
Then, there exists a solution $u$ of \eqref{eq:nls-3D} on $\R$ with $w\in C(\R;H_x^{1}(\R^3))$. Furthermore, there exists $u_{\pm}\in H_x^1(\R^3)$ such that
\EQ{
\lim_{t\ra\pm\I}\norm{u(t)-v(t)-e^{it\De}u_{\pm}}_{H_x^1(\R^3)}=0.
}
\end{prop}
We will give the proof of Proposition \ref{prop:global-derterministic} in Sections \ref{sec:local}-\ref{sec:global-scattering}. Now we prove Theorem \ref{thm:global} in 3D case assuming that Proposition \ref{prop:global-derterministic} holds.
\begin{proof}[Proof of Theorem \ref{thm:global}]
Now, we only need to prove the $s<0$ case. In fact, in the case when $0<s<1$, the mass conservation law is available, so the proof of related result in Proposition \ref{prop:global-derterministic} is easier. 

For any $s<0$, by the Definition \ref{defn:randomization}, we can find a randomization $f^\om$ with the parameter $a$. Furthermore, $a$ and $s$ satisfy the assumption of Proposition \ref{prop:global-derterministic}. Let $N_0\in 2^\N$ to be defined later, and make a high-low frequency decomposition for the initial data
\EQ{
	u(t)=e^{it\De}P_{\goe N_0}f^\om + w(t),
}
then $w$ satisfies the equation \eqref{eq:nls-w} with
\EQ{
	u_0=f^\om\text{, }v_0=P_{\goe N_0}f^\om\text{, }w_0=P_{\loe N_0 }f^\om\text{, and }v=e^{it\De}P_{\goe N_0}f^\om.
}

By Corollary \ref{cor:Y-norm-Z-norm} and boundedness of the operator $P_{\goe N_0}$, we have
\EQn{\label{esti:omM-1}
\PP\brkb{\fbrkb{\om\in\Om:\norm{u_0}_{H_x^s}+ \norm{v}_{Y\cap Z(\R)}>\la}}\lsm e^{-C\la^2\norm{f}_{H_x^s}^{-2}}.
}
For any $p\goe 2$, by Lemmas \ref{lem:large-deviation} and \ref{lem:orthogonality},
\EQ{
	\norm{w_0}_{L_\om^p L_x^2}\lsm \sqrt{p} \norm{\bx_jP_{\loe N_0}f}_{L_x^2 l_j^2}\lsm\sqrt{p} \norm{P_{\loe N_0}f}_{L_x^2} \lsm \sqrt{p} N_0^{-s}\norm{f}_{H_x^s}.
}
Then, by Lemma \ref{lem:probability-estimate}, we have
\EQn{\label{esti:omM-2}
\PP\brkb{\fbrkb{\om\in\Om: N_0^s\norm{w_0}_{ L_x^2}>\la}}\lsm e^{-C\la^2\norm{f}_{H_x^s}^{-2}}.
}
For any $p\goe 2$, by Lemmas \ref{lem:large-deviation} and \ref{lem:orthogonality},
\EQ{
\norm{w_0}_{L_\om^p \dot H_x^1}
\lsm & \sqrt{p}\norm{\nabla P_{\loe N_0}\bx_jf}_{L_x^2 l_{k\in\N^+}^2 } \\
\lsm & \sqrt{p}\norm{\nabla P_{\loe N_0}f}_{L_x^2}\lsm \sqrt{p} N_0^{1-s} \norm{f}_{H_x^s}.
}
For any $p\goe 6$, since $s>-a$, by Minkowski's inequality, Lemmas \ref{lem:large-deviation}, \ref{lem:bernstein}, and  \ref{lem:orthogonality},
\EQ{
	\norm{u_0}_{L_\om^p L_x^6}\lsm \sqrt{p} \norm{\bx_jf}_{L_x^6 l_j^2}\lsm\sqrt{p} \norm{\jb{\nabla}^{-a}f}_{L_x^2} \lsm \sqrt{p} \norm{f}_{H_x^s}.
}
Note that $N_0$ only depends on $M$ and $\norm{f}_{H_x^s}$. Then, by Lemma \ref{lem:probability-estimate},
\EQn{\label{esti:omM-3}
\PP\brkb{\fbrkb{\om\in\Om: N_0^{s-1}\norm{w_0}_{ \dot H_x^1} + \norm{u_0}_{L_x^6}\goe\la}} \lsm e^{-C\la^2\norm{f}_{H_x^s}^{-2}}.
}

For any $M\goe 1$, let $\wt \Om_M$ be defined by
\EQn{\label{defn:omega-M}
\wt \Om_M=\Big\{\om\in\Om: &\norm{u_0}_{H_x^s} + \norm{v}_{ Y\cap Z(\R)} < M \norm{f}_{H_x^s},\\
&N_0^s\norm{w_0}_{ L_x^2} +N_0^{s-1}\norm{w_0}_{ \dot H_x^1} + \norm{u_0}_{L_x^6}<M\norm{f}_{H_x^s}\Big\}.
}
Therefore, by \eqref{esti:omM-1} and \eqref{esti:omM-2}, we have
\EQn{\label{Om-Mc}
\PP(\wt \Om_M^c) \lsm e^{-CM^2}.
}
For any $\om \in\wt \Om_M$, we have $\norm{v}_{Y\cap Z(\R)}< M \norm{f}_{H_x^s}$, 
\EQ{
M(0)\loe C N_0^{-2s} M^2 \norm{f}_{H_x^s}^2\text{, and }E(0)\loe C N_0^{2(1-s)}M^2\cdot\max\fbrkb{M^4\norm{f}_{H_x^s}^6,\norm{f}_{H_x^s}^2}.
}

Therefore, for any $M>1$ and any $\om\in\wt \Om_M$, let 
\EQ{
A=A(M,\norm{f}_{H_x^s}):=\max\fbrkb{CM \norm{f}_{H_x^s},CM^2\cdot\max\fbrko{M^4\norm{f}_{H_x^s}^6,\norm{f}_{H_x^s}^2}},
}
then we have for $v=e^{it\De}P_{\goe N_0}f^\om$,
\EQ{
\norm{u_0}_{H_x^s} +\norm{v}_{ Y\cap Z(\R)}\loe A\text{, }M(0)\loe AN_0^{-2s}\text{, and }E(0)\loe AN_0^{2(1-s)}.
} 
Therefore, we can apply Proposition \ref{prop:global-derterministic}. Let $N_0$ depend on $A$ as in the statement of Proposition \ref{prop:global-derterministic}, and we obtain a global solution $w$ that scatters. Then, for any $\om\in\wt \Om=\cup_{M>1}\wt \Om_M$, we can also derive that \eqref{eq:nls-w} admits a global solution $w$ that scatters. By \eqref{Om-Mc},  we have that  $\PP(\wt \Om)=1$. Then for almost every $\om\in\Om$, we obtain the global well-posedness and scattering for \eqref{eq:nls-w}. This finishes the proof of Theorem \ref{thm:global} in 3D case.
\end{proof}

\subsection{Local theory}\label{sec:local}
We define the space $X(I)$ as
\EQ{
\norm{w}_{X(I)}:=\norm{\jb{\nabla}w}_{L_t^2 L_x^{6}(I\times \R^3)} + \norm{w}_{L_{t,x}^8(I\times \R^3)} + \norm{w}_{L_t^8 L_x^{12}(I\times \R^3)}.
}
\begin{lem}[Local well-posedness]\label{lem:local}
Let $a\in\N$, $a> 10$, $\frac34-\frac14a<s<0$, $v\in Y\cap Z(\R)$, and $w_0\in H_x^1$. Then, there exists some $T>0$ depending on $a$, $w_0$, and $v$ such that there exists a unique solution $w$ of \eqref{eq:nls-w} in some $0$-neighbourhood of 
\EQ{
C([0,T];H_x^1(\R^3)) \cap X([0,T]).
}
\end{lem}
\begin{proof}
First, we make the choices of some parameters:
\enu{
\item Let $C_0$ be the constant such that
\EQ{
\norm{e^{it\De}w_0}_{L_t^\I(\R; H_x^1)\cap X(\R)}\loe C_0\norm{w_0}_{H_x^1}.
}
\item 
Define
\EQ{
R:=\max\fbrk{C_0\norm{w_0}_{H_x^1},1}.
}
\item 
Let $\de>0$ be some small constant such that $C\de^4 R^4\loe \frac12$.
\item 
Let $T>0$ satisfy the smallness condition
\EQ{
\norm{e^{it\De}w_0}_{X([0,T])} + \norm{v}_{Y([0,T])}\loe\de R.
}
}
Now, we define the working space as
\EQ{
B_{R,\de,T}:=\fbrk{w\in C([0,T];H_x^1):\norm{w}_{L_t^\I H_x^1([0,T]\times\R^3)} + \de^{-1}\norm{w}_{X([0,T])}\loe 4 R},
}
equipped with the norm
\EQ{
\norm{w}_{B_{R,\de,T}}:= \norm{w}_{L_t^\I H_x^1([0,T]\times\R^3)} + \de^{-1}\norm{w}_{X([0,T])}.
}
Take the solution map as
\EQ{
\Phi_{w_0,v}(w):=e^{it\De}w_0-i\int_0^t e^{i(t-s)\De}(|u|^4u)\ds.
}
Then, it suffices to prove that $\Phi_{w_0,v}$ is a contraction mapping on $B_{R,\de,T}$.

Next, we prove that for any $w\in B_{R,\de,T}$, $\Phi_{w_0,v}(w)\in B_{R,\de,T}$. By Lemma \ref{lem:strichartz} and H\"older's inequality,
\EQ{
\norm{\Phi_{w_0,v}(w)}_{L_t^\I H_x^1} \loe & C_0\norm{w_0}_{H_x^1} + C\norm{\nabla (|u|^4u)}_{L_t^1 L_x^2} + C\norm{|u|^4u}_{L_t^1 L_x^2} \\
\loe & R + C \norm{\nabla w u^4}_{L_t^1 L_x^2} + C \norm{\nabla v u^4}_{L_t^1 L_x^2} + C\norm{|u|^4u}_{L_t^1 L_x^2} \\
\loe & R + C\norm{\nabla w}_{L_t^2 L_x^6}\norm{u}_{L_t^8 L_x^{12}}^4 + C\norm{\nabla v}_{L_t^2 L_x^\I} \norm{u}_{L_{t,x}^8}^4 \\
&+ C\norm{w}_{L_t^2 L_x^6}\norm{u}_{L_t^8 L_x^{12}}^4 + C\norm{v}_{L_t^2 L_x^\I}\norm{u}_{L_{t,x}^8}^4.
}
Note that $v$ is high frequency part, then by the definition of $X$ and $Y$ norms,
\begin{gather}
\norm{\jb{\nabla} w}_{L_t^2 L_x^6} \loe \norm{w}_{X}\lsm \de R,\nonumber\\
\norm{\jb{\nabla} v}_{L_t^2 L_x^\I}\lsm \norm{v}_{Y}\lsm \de R,\nonumber\\
\norm{u}_{L_t^8 L_x^{12}} \loe \norm{w}_{L_t^8 L_x^{12}} + \norm{v}_{L_t^8 L_x^{12}}\loe \norm{w}_{X} + \norm{v}_{Y} \lsm \de R,\nonumber
\end{gather}
and
\EQ{
\norm{u}_{L_{t,x}^8} \loe \norm{w}_{L_{t,x}^8} + \norm{v}_{L_{t,x}^8}\loe \norm{w}_{X} + \norm{v}_{Y} \lsm \de R.
}
Therefore, by the choice of $\de$,
\EQn{\label{esti:local-bound-1}
\norm{\Phi_{w_0,v}(w)}_{L_t^\I H_x^1} \loe & R + C\de^5 R^5 \loe 2R.
}
Similar as above, we also have
\EQn{\label{esti:local-bound-2}
\norm{\Phi_{w_0,v}(w)}_{X} \loe & \norm{e^{it\De}w_0}_{H_x^1} + C\norm{\nabla (|u|^4u)}_{L_t^1 L_x^2} + C\norm{|u|^4u}_{L_t^1 L_x^2} \\
\loe & \de R + C\de^5 R^5 \loe2\de R.
}
Then, by \eqref{esti:local-bound-1} and \eqref{esti:local-bound-1},
\EQ{
\norm{\Phi_{w_0,v}(w)}_{B_{R,\de,T}} \loe \norm{\Phi_{w_0,v}(w)}_{L_t^\I H_x^1} + \de^{-1}\norm{\Phi_{w_0,v}(w)}_{X}\loe 4R.
}
This shows that $\Phi_{w_0,v}$  maps $B_{R,\de,T}$ into itself.

Next, we are going to prove that  $\Phi_{w_0,v}$ is a contraction mapping. Take $w_1$ and $w_2$ in $B_{R,\de,T}$. By Lemma \ref{lem:strichartz},
\EQn{\label{esti:local-difference-1}
&\norm{\Phi_{w_0,v}(w_1)-\Phi_{w_0,v}(w_2)}_{L_t^\I H_x^1\cap X}\\
\loe & C\normb{\jb{\nabla}(|w_1+v|^4(w_1+v) - |w_2+v|^4(w_2+v))}_{L_t^1 L_x^2}.
}
Note that we have an elementary inequality
\EQ{
||w_1+v|^4(w_1+v) - |w_2+v|^4(w_2+v)|\lsm|w_1-w_2|(|w_1|^4 + |w_2|^4 + |v|^4).
}
Then, by H\"older's inequality,
\EQn{\label{esti:local-difference-2}
&\normb{|w_1+v|^4(w_1+v) - |w_2+v|^4(w_2+v)}_{L_t^1 L_x^2}\\
\lsm & \norm{w_1-w_2}_{L_t^2 L_x^6} (\norm{w_1}_{L_t^8 L_x^{12}}^4 + \norm{w_2}_{L_t^8 L_x^{12}}^4 + \norm{v}_{L_t^8 L_x^{12}}^4)\\
\lsm & \de \norm{w_1-w_2}_{B_{R,\de,T}}\cdot \de^4 R^4.
}
Note that we also have
\EQ{
&|\nabla(|w_1+v|^4(w_1+v) -  |w_2+v|^4(w_2+v))|\\
\lsm& |\nabla(w_1-w_2)|(|w_1|^4  + |v|^4)
 + |w_1-w_2|(|\nabla w_2|+|\nabla v|)(|w_1|^4 + |w_2|^4 + |v|^4).
}
Then, by H\"older's inequality,
\EQn{\label{esti:local-difference-3}
&\normb{\nabla(|w_1+v|^4(w_1+v) -  |w_2+v|^4(w_2+v))}_{L_t^1 L_x^2}\\
\lsm & \norm{\nabla(w_1-w_2)}_{L_t^2 L_x^6}(\norm{w_1}_{L_t^8 L_x^{12}}^4 + \norm{v}_{L_t^8 L_x^{12}}^4) \\
& + \norm{w_1-w_2}_{L_t^8 L_x^{12}} \norm{\nabla w_2}_{L_t^2 L_x^6} (\norm{w_1}_{L_t^8 L_x^{12}}^4 + \norm{w_2}_{L_t^8 L_x^{12}}^4 + \norm{v}_{L_t^8 L_x^{12}}^4) \\
& + \norm{w_1-w_2}_{L_{t,x}^8} \norm{\nabla v}_{L_t^2 L_x^\I} (\norm{w_1}_{L_{t,x}^8}^4 + \norm{w_2}_{L_{t,x}^8}^4 + \norm{v}_{L_{t,x}^8}^4) \\
\lsm & \de^5 R^4 \norm{w_1-w_2}_{B_{R,\de,T}}.
}
Therefore, by \eqref{esti:local-difference-1}, \eqref{esti:local-difference-2}, and \eqref{esti:local-difference-3},
\EQn{\label{esti:local-difference-4}
\norm{\Phi_{w_0,v}(w_1)-\Phi_{w_0,v}(w_2)}_{L_t^\I H_x^1\cap X}\loe C \de^5 R^4 \norm{w_1-w_2}_{B_{R,\de,T}} .
}
Then, by \eqref{esti:local-difference-4} and the choice of $\de$,
\EQ{
\norm{\Phi_{w_0,v}(w_1)-\Phi_{w_0,v}(w_2)}_{B_{R,\de,T}}\loe & C \de^5 R^4 \norm{w_1-w_2}_{B_{R,\de,T}} + C \de^4 R^4 \norm{w_1-w_2}_{B_{R,\de,T}} \\
\loe & \frac12\norm{w_1-w_2}_{B_{R,\de,T}}.
}
This proves that $\Phi_{w_0,v}$ is a contraction mapping on $B_{R,\de,T}$.
\end{proof}

\subsection{Modified Interaction Morawetz}\label{sec:space-time}
We need to prove a perturbation version of interaction Morawetz estimate as follows.
\begin{lem}[Modified Interaction Morawetz]\label{lem:inter-mora}
Given $T>0$. Let $w\in C([0,T];H_x^1)$ be the solution of perturbation equation \eqref{eq:nls-w}. Then, we have
\EQn{\label{eq:inter-mora}
\norm{w}_{L_{t,x}^4}^4\lsm & \norm{w}_{L_t^\I L_x^2}^2 \norm{w}_{L_t^\I \dot H_x^{\frac12}}^2 \\
&+ \norm{v}_{L_t^2 L_x^\I} (\norm{w}_{L_{t,x}^4}^2 + \norm{v}_{L_{t,x}^4}^2)(\norm{w}_{L_t^\I L_x^6}^3 + \norm{v}_{L_t^\I L_x^6}^3) \norm{w}_{L_t^\I \dot H_x^{\frac12}}^2 \\
& + \norm{\nabla v}_{L_t^2 L_x^\I} (\norm{w}_{L_{t,x}^4}^2 + \norm{v}_{L_{t,x}^4}^2)(\norm{w}_{L_t^\I L_x^6}^3 + \norm{v}_{L_t^\I L_x^6}^3) \norm{w}_{L_t^\I L_x^2}^2,
}
where all the space-time norms are taken over $[0,T]\times \R^3$.
\end{lem}
\begin{proof}
Recall that $w$ satisfies
\EQ{
i\pd_t w + \De w=|w|^4w+e,
}
where we denote $e:=|u|^4u-|w|^4w$.
Denote that 
$$
m(t,x)=\frac12|w(t,x)|^2;\quad 
p(t,x)=\frac12\im \brko{\wb w(t,x)\nabla w(t,x)}.
$$
Then, we have
\EQn{\label{eq:local-mass-flow}
\pd_tm=-2\nabla\cdot p+ \im\brk{e\bar w},
}
and
\EQn{\label{eq:local-momentum-flow}
\pd_tp = & -\re \nabla\cdot \brko{\nabla \wb w \nabla w}  -\frac16 \nabla \brkb{|w|^6} + \half 1 \nabla \De m
+ \re \brk{\wb e\nabla w}-\frac12\re\nabla\brk{\wb w e}.
}
Moreover, we note that 
$$
\partial_j\Big(\frac{x_k}{|x|}\Big)=\frac{\delta_{jk}}{|x|}-\frac{x_jx_k}{|x|^3};\quad
\nabla\cdot \frac{x}{|x|}=\frac2{|x|};\quad \De\nabla\cdot \frac{x}{|x|}=\delta(x).
$$
Let 
\EQ{
M(t):= \int\!\!\int_{\R^{3+3}} \frac{x-y}{|x-y|}\cdot p(t,x)\> m(t,y)\dx\dy,
}
then by \eqref{eq:local-mass-flow} and \eqref{eq:local-momentum-flow}, we have the interaction Morawetz identity
\begin{subequations}\label{4.18-a-f}
\EQnn{
 \pd_t M(t)
= &\iint_{\R^{3+3}} \frac{x-y}{|x-y|}\cdot \partial_t p(t,x)\> m(t,y)\dx\dy\nonumber\\
& \quad + \iint_{\R^{3+3}} \frac{x-y}{|x-y|}\cdot p(t,x)\> \partial_t m(t,y)\dx\dy\nonumber\\
= &\iint_{\R^{3+3}} \frac{x-y}{|x-y|}\cdot \Big( -\re \nabla\cdot \brko{\nabla \wb w \nabla w}  -\frac16 \nabla \brkb{|w|^6}\Big)(t,x)\> m(t,y)\dx\dy\label{esti:inter-mora-angular-1}\\
& \quad -2\iint_{\R^{3+3}} \frac{x-y}{|x-y|}\cdot p(t,x)\> \nabla\cdot p(t,y)\dx\dy\label{esti:inter-mora-angular-2}\\
&\quad +\frac12\iint_{\R^{3+3}} \frac{x-y}{|x-y|}\cdot \nabla \De m(t,x)\> m(t,y)\dx\dy\label{esti:inter-mora-angular-3}\\
&\quad +\iint_{\R^{3+3}} \frac{x-y}{|x-y|}\cdot  p(t,x)\>  \im\brk{e\bar w}(t,y)\dx\dy\label{esti:inter-mora-remainder-1}\\
& \quad + \iint_{\R^{3+3}} \frac{x-y}{|x-y|}\cdot \re \brk{\wb e\nabla w}(t,x)\> m(t,y)\dx\dy\label{esti:inter-mora-remainder-2}\\
& \quad + \iint_{\R^{3+3}} \frac1{|x-y|}\cdot \re \brk{\wb e w}(t,x)\> m(t,y)\dx\dy.\label{esti:inter-mora-remainder-3}
}
\end{subequations}
Note that by the classical argument in \cite{CKSTT04CPAM}, we have
\EQ{
\eqref{esti:inter-mora-angular-1} + \eqref{esti:inter-mora-angular-2} \goe 0,
}
and
\EQ{
\eqref{esti:inter-mora-angular-3}\gsm \norm{w(t)}_{L_x^4}^4.
}
Moreover,
\EQ{
\sup_{t\in[0,T]} M(t)\lsm \norm{w}_{L_t^\I L_x^2}^2 \norm{w}_{L_t^\I \dot H_x^{\frac12}}^2.
}
Then, integrating over $[0,T]$, it holds that
\EQ{
C\norm{w}_{L_{t,x}^4}^4\loe M(T)-M(0)+ \int_0^T |\eqref{esti:inter-mora-remainder-1}| + |\eqref{esti:inter-mora-remainder-2}| +|\eqref{esti:inter-mora-remainder-3}| \dt,
}
thus
\EQn{\label{esti:inter-mora-bound-1}
\norm{w}_{L_{t,x}^4}^4\lsm \norm{w}_{L_t^\I L_x^2}^2 \norm{w}_{L_t^\I \dot H_x^{\frac12}}^2+ \int_0^T |\eqref{esti:inter-mora-remainder-1}| + |\eqref{esti:inter-mora-remainder-2}| +|\eqref{esti:inter-mora-remainder-3}|\dt.
}
Next, we estimate the terms containing \eqref{esti:inter-mora-remainder-1}, \eqref{esti:inter-mora-remainder-2}, and \eqref{esti:inter-mora-remainder-3}. We first consider \eqref{esti:inter-mora-remainder-1}. By H\"older's inequality,
\EQ{ 
\int_0^T |\eqref{esti:inter-mora-remainder-1}| \dt 
\lsm & \int_0^T \absb{\iint_{\R^{3+3}} \frac{x-y}{|x-y|}\cdot  p(t,x)\>  \im\brk{e\bar w}(t,y)\dx\dy} \dt \\
\lsm & \int_0^T \absb{\int_{\R^{3}}   \im\brk{e\bar w}(t,y)\dy} \>  \sup_{y}\absb{\int_{\R^3}\frac{x-y}{|x-y|}\cdot  p(t,x)\dx} \dt\\
\lsm &\|e\|_{L_t^1L^\frac65_x}\| w\|_{L_t^\I L^6_x} \norm{w}_{L_t^\I \dot H_x^{\frac12}}^2.
}
Note that 
$$
\|e\|_{L_t^1L^\frac65_x}\lesssim \|v\|_{L^2_t L^\infty_x} (\norm{w}_{L_{t,x}^4}^2 + \norm{v}_{L_{t,x}^4}^2)(\norm{w}_{L_t^\I L_x^6}^2 + \norm{v}_{L_t^\I L_x^6}^2).
$$
Hence, we get that 
\eqref{esti:inter-mora-remainder-3}. We first consider \eqref{esti:inter-mora-remainder-1}. By H\"older's inequality,
\EQ{ 
\int_0^T |\eqref{esti:inter-mora-remainder-1}| \dt 
\lesssim 
\|v\|_{L^2_t L^\infty_x} (\norm{w}_{L_{t,x}^4}^2 + \norm{v}_{L_{t,x}^4}^2)(\norm{w}_{L_t^\I L_x^6}^3 + \norm{v}_{L_t^\I L_x^6}^3)
\norm{w}_{L_t^\I \dot H_x^{\frac12}}^2.
}

We then consider \eqref{esti:inter-mora-remainder-2}, where we need to modify the Morawetz estimate:
\EQ{
\int_0^T |\eqref{esti:inter-mora-remainder-2}| \dt \lsm & \int_0^T \absb{\iint_{\R^{3+3}} \frac{x-y}{|x-y|}\cdot \re \brk{(|u|^4u-|w|^4w)\nabla \wb w}(t,x)\> m(t,y)\dx\dy} \dt.
}
We note that
\EQ{
(|u|^4u-|w|^4w)\nabla \wb w(t,x) = vw^4\nabla \wb w(t,x) + \text{other terms}.
}
However,  it is difficult to estimate the piece $vw^4\nabla \wb w(t,x)$. To this end, we need the following equality;
\EQn{\label{Mor-Mainterm}
&\int_{\R^3} \frac{x-y}{|x-y|}\cdot \re \brk{(|u|^4u-|w|^4w)\nabla \wb w}(t,x) \dx \\
= & -\frac13\int_{\R^3}  \frac{1}{|x-y|}\brk{|u(t,x)|^6-|w(t,x)|^6}\dx - \int_{\R^3}  \frac{x-y}{|x-y|}\cdot \re \brk{|u|^4u\nabla \wb v}(t,x)\dx.
}
Indeed, 
\EQ{
&\int_{\R^3} \frac{x-y}{|x-y|}\cdot \re \brk{(|u|^4u-|w|^4w)\nabla \wb w}(t,x) \dx \\
=& \int_{\R^3}  \frac{x-y}{|x-y|}\cdot \re \brk{|u|^4u\nabla \wb w}(t,x) \dx - \int_{\R^3}  \frac{x-y}{|x-y|}\cdot \re \brk{|w|^4w\nabla \wb w}(t,x) \dx\\
= & \int_{\R^3}  \frac{x-y}{|x-y|}\cdot \re \brk{|u|^4u\nabla (\wb u-\wb v)}(t,x)\dx - \frac16\int_{\R^3}  \frac{x-y}{|x-y|}\cdot  \nabla\brk{|w(t,x)|^6}\dx \\
=& \frac16\int_{\R^3}  \frac{x-y}{|x-y|}\cdot  \nabla\brk{|u(t,x)|^6-|w(t,x)|^6}\dx - \int_{\R^3}  \frac{x-y}{|x-y|}\cdot \re \brk{|u|^4u\nabla \wb v}(t,x)\dx \\
= & -\frac13\int_{\R^3}  \frac{1}{|x-y|}\brk{|u(t,x)|^6-|w(t,x)|^6}\dx - \int_{\R^3}  \frac{x-y}{|x-y|}\cdot \re \brk{|u|^4u\nabla \wb v}(t,x)\dx.
}
This gives \eqref{Mor-Mainterm}. 
Therefore, by \eqref{Mor-Mainterm} and Lemma \ref{lem:hardy}, we have
\EQ{ 
\int_0^T |\eqref{esti:inter-mora-remainder-2}| \dt \lsm & \int_0^T \absb{\iint_{\R^{3+3}} \frac{1}{|x-y|}\brk{|u(t,x)|^6-|w(t,x)|^6}\> m(t,y)\dx\dy} \dt \\
& + \int_0^T \absb{\iint_{\R^{3+3}} \frac{x-y}{|x-y|}\cdot \re \brk{|u|^4u\nabla \wb v}(t,x)\> m(t,y)\dx\dy} \dt\\
\lsm & \int_0^T \absb{\int_{\R^{3}} \brk{|u(t,x)|^6-|w(t,x)|^6}\dx} \sup_{x}\normb{\frac{1}{|x-\cdot|^{\frac12}}w(t,\cdot)}_{L_y^2}^2 \dt \\
& + \int_0^T \int_{\R^{3}}  \absb{|u|^4u\nabla \wb v(t,x)}\dx\norm{w(t)}_{L_y^2}^2 \dt\\
\lsm & \norm{v}_{L_t^2 L_x^\I} (\norm{w}_{L_{t,x}^4}^2 + \norm{v}_{L_{t,x}^4}^2)(\norm{w}_{L_t^\I L_x^6}^3 + \norm{v}_{L_t^\I L_x^6}^3) \norm{w}_{L_t^\I \dot  H_x^{\frac12}}^2 \\
& + \norm{\nabla v}_{L_t^2 L_x^\I} (\norm{w}_{L_{t,x}^4}^2 + \norm{v}_{L_{t,x}^4}^2)(\norm{w}_{L_t^\I L_x^6}^3 + \norm{v}_{L_t^\I L_x^6}^3) \norm{w}_{L_t^\I L_x^2}^2.
}

Now, we consider the term  \eqref{esti:inter-mora-remainder-3}. Treated similarly as above, we have that 
\EQ{
	\int_0^T |\eqref{esti:inter-mora-remainder-3}| \dt \lsm & \int_0^T \absb{\iint_{\R^{3+3}} \frac{1}{|x-y|} e\wb w(t,x)\> m(t,y)\dx\dy} \dt \\
	\lsm & \int_0^T \int_{\R^{3}} \abs{e\wb w(t,x)}\dx\> \sup_{x}\normb{\frac{1}{|x-\cdot|^{\frac12}}w(t,\cdot)}_{L_y^2}^2 \dt \\
	\lsm & \norm{v}_{L_t^2 L_x^\I} (\norm{w}_{L_{t,x}^4}^2 + \norm{v}_{L_{t,x}^4}^2)(\norm{w}_{L_t^\I L_x^6}^3 + \norm{v}_{L_t^\I L_x^6}^3) \norm{w}_{L_t^\I \dot  H_x^{\frac12}}^2.
}

Combining the findings on \eqref{4.18-a-f}, we have that 
\EQn{\label{esti:inter-mora-remainder-bound}
&\int_0^T|\eqref{esti:inter-mora-remainder-1}| + |\eqref{esti:inter-mora-remainder-2}| +|\eqref{esti:inter-mora-remainder-3}|\dt\\
\lsm & \norm{v}_{L_t^2 L_x^\I} (\norm{w}_{L_{t,x}^4}^2 + \norm{v}_{L_{t,x}^4}^2)(\norm{w}_{L_t^\I L_x^6}^3 + \norm{v}_{L_t^\I L_x^6}^3) \norm{w}_{L_t^\I \dot H_x^{\frac12}}^2 \\
& + \norm{\nabla v}_{L_t^2 L_x^\I} (\norm{w}_{L_{t,x}^4}^2 + \norm{v}_{L_{t,x}^4}^2)(\norm{w}_{L_t^\I L_x^6}^3 + \norm{v}_{L_t^\I L_x^6}^3) \norm{w}_{L_t^\I L_x^2}^2.
}

Therefore, by \eqref{esti:inter-mora-bound-1} and \eqref{esti:inter-mora-remainder-bound},
\EQ{
\norm{w}_{L_{t,x}^4}^4\lsm & \norm{w}_{L_t^\I L_x^2}^2 \norm{w}_{L_t^\I \dot H_x^{\frac12}}^2 \\
&+ \norm{v}_{L_t^2 L_x^\I} (\norm{w}_{L_{t,x}^4}^2 + \norm{v}_{L_{t,x}^4}^2)(\norm{w}_{L_t^\I L_x^6}^3 + \norm{v}_{L_t^\I L_x^6}^3) \norm{w}_{L_t^\I \dot H_x^{\frac12}}^2 \\
& + \norm{\nabla v}_{L_t^2 L_x^\I} (\norm{w}_{L_{t,x}^4}^2 + \norm{v}_{L_{t,x}^4}^2)(\norm{w}_{L_t^\I L_x^6}^3 + \norm{v}_{L_t^\I L_x^6}^3) \norm{w}_{L_t^\I L_x^2}^2.
}
This completes the proof of this lemma.
\end{proof}

\subsection{Almost conservation law}\label{sec:energy-bound}
Our main result in this subsection is 
\begin{prop}\label{prop:energy-bound}
Let $a\in\N$, $a> 10$, $\frac34-\frac14a<s<0$, $A>0$, $v=e^{it\De}v_0\in  Y\cap Z(\R)$ and $w$ be the solution of \eqref{eq:nls-w}. Take some $T>0$ such that $w\in C([0,T];H_x^1)$. Then, there exists $N_0=N_0(A)\gg1$ with the following properties. Assume that $\wh v_0$ is supported on $\fbrk{\xi\in\R^3:|\xi|\goe \frac12 N_0}$,
\EQ{
	\norm{u_0}_{H_x^s}+\norm{v}_{Y\cap Z(\R)}\loe A\text{, }M(0)\loe AN_0^{-2s}\text{, and }E(0) \loe A N_0^{2(1-s)}.
}
Then, we have
\EQn{\label{eq:energy-bound}
\sup_{t\in[0,T]}M(t)\loe 2A N_0^{-2s}\text{, and }\sup_{t\in[0,T]}E(t)\loe 2A N_0^{2(1-s)}.
}
\end{prop}
\begin{proof}
Let $N_0=N_0(A)$ that will be defined later. We implement a bootstrap procedure on $I\subset [0,T]$: assume an a priori bound 
\EQn{\label{eq:bound-w-hypothesis}
\sup_{t\in I}M(t)\loe 2A N_0^{-2s}\text{, and }\sup_{t\in I} E(t)\loe 2AN_0^{2(1-s)},
} 
then it suffices to prove that
\EQn{\label{eq:bound-w-bootstrap}
\sup_{t\in I}M(t)\loe \frac32A N_0^{-2s}\text{, and }\sup_{t\in I} E(t)\loe \frac32AN_0^{2(1-s)}.
}
From now on, all the space-time norms are taken over $I\times\R^3$. 

To start with, we collect useful estimates on $I$. Now, we use the notation $C=C(A)$ for short, and  the implicit constants in ``$\lesssim$'' depend on $A$. By Lemma \ref{lem:GN-2}, we have $Z\subset L_t^\I L_x^6$, then
\EQn{\label{eq:bound-v-ltinfty}
\norm{v}_{L_t^\I L_x^6} + \norm{v}_{L_{t,x}^4}\lsm\norm{v}_{Y\cap Z}\lsm1. 
}
By the frequency support of $v$, we have for any $0\loe l<\frac12a+s$,
\EQn{\label{eq:bound-v-l2linfty}
	\normb{|\nabla|^lv}_{L_t^2 L_x^\I}\lsm N_0^{l-\frac12a-s+}\norm{v}_{ Y}\lsm N_0^{l-\frac12a-s+}\lsm 1.
}
Note that we assume $a>10$ and $s>\frac34-\frac14a$, then $\frac12 a + s > 2 $. Therefore, this guarantees that $\norm{\De v}_{L_t^2 L_x^\I} \lsm N_0^{2-\frac12a-s+}$ is allowed.
By bootstrap hypothesis \eqref{eq:bound-w-hypothesis},
\EQn{\label{eq:bound-w-h1}
\norm{w}_{L_t^\I L_x^2}\lsm N_0^{-s}\text{, }\norm{w}_{L_t^\I \dot H_x^1}\lsm N_0^{1-s}\text{, and } \norm{w}_{L_t^\I L_x^6}\lsm N_0^{\frac13(1-s)}.
}
Then, by interpolation and \eqref{eq:bound-w-h1}, we have for any $0\loe l \loe 1$, 
\EQn{\label{eq:bound-w-ltinfty}
\norm{w}_{L_t^\I \dot H_x^l}\lsm N_0^{l-s}.
}
Furthermore, by Lemma \ref{lem:inter-mora}, \eqref{eq:bound-w-ltinfty}, and \eqref{eq:bound-v-l2linfty},
\EQ{
	\norm{w}_{L_{t,x}^4}^4\lsm & \norm{w}_{L_t^\I L_x^2}^2 \norm{w}_{L_t^\I \dot H_x^{\frac12}}^2 \\
	&+ \norm{v}_{L_t^2 L_x^\I} (\norm{w}_{L_{t,x}^4}^2 + \norm{v}_{L_{t,x}^4}^2)(\norm{w}_{L_t^\I L_x^6}^3 + \norm{v}_{L_t^\I L_x^6}^3) \norm{w}_{L_t^\I \dot H_x^{\frac12}}^2 \\
	& + \norm{\nabla v}_{L_t^2 L_x^\I} (\norm{w}_{L_{t,x}^4}^2 + \norm{v}_{L_{t,x}^4}^2)(\norm{w}_{L_t^\I L_x^6}^3 + \norm{v}_{L_t^\I L_x^6}^3) \norm{w}_{L_t^\I L_x^2}^2\\
	\lsm & N_0^{1-4s} + N_0^{-\frac12a-s+} (\norm{w}_{L_{t,x}^4}^2 + 1) (N_0^{1-s}+1) N_0^{1-2s} \\
	&+ N_0^{1-\frac12a-s+} (\norm{w}_{L_{t,x}^4}^2 + 1) (N_0^{1-s}+1)N_0^{-2s}\\
	\lsm & N_0^{1-4s}+ N_0^{2-\frac12a-4s+} + N_0^{2-\frac12a-4s+}\norm{w}_{L_{t,x}^4}^2.
}
By $s>\frac34-\frac14 a$ and Young's inequality, 
\EQn{\label{eq:bound-w-ltx4}
\norm{w}_{L_{t,x}^4}^4\lsm N_0^{1-4s}+ N_0^{2-\frac12a-4s+} + N_0^{4-a-8s+} \lsm N_0^{1-4s}.
}

Now, we are prepared to give the proof of \eqref{eq:bound-w-bootstrap}.  To do this, we first need the following lemma. 
\begin{lem}\label{lem:conservation-law}
Assume that $w\in C(I; H_x^1(\R^3))$ solves \eqref{eq:nls-w}. Let $E(t)$ and $M(t)$ be defined as in \eqref{defn:energy-w} and \eqref{defn:mass-w}. Then, for any $t\in I$,
\EQ{
\absb{\frac{\mathrm{d}}{\mathrm{d}t}M(t)} \loe 2\absb{\int_{\R^3} \brkb{|u|^4u-|w|^4w} \wb w\dx}\text{, and  }\absb{\frac{\mathrm{d}}{\mathrm{d}t}E(t)} \loe \absb{\int_{\R^3} |u|^4u\De \wb v\dx}.
}
\end{lem}
\begin{proof}
First, by \eqref{eq:nls-w} and integration-by-parts,
\EQ{
\frac{\mathrm{d}}{\mathrm{d}t}\brk{\int\abs{w(t,x)}^2\dx} = & 2\re\int \wb w w_t\dx \\
= & 2\re i\int \wb w \brkb{\De w-|u|^4u}\dx \\
= & -2\re i\int \wb w \brkb{|u|^4u - |w|^4w}\dx.
}
Similarly by \eqref{eq:nls-w} and integration-by-parts,
\EQ{
	\frac{\mathrm{d}}{\mathrm{d}t}\brk{\half 1\int\abs{\nabla w(t,x)}^2\dx}=&-\Re\int \De w \wb{w}_t\dx\\
	=&\Re i\int w_t\wb{w}_t\dx-\Re\int |u|^4u\wb w_t\dx\\
	=&-\Re\int |u|^4u\wb u_t\dx +\Re\int |u|^4u\wb v_t\dx\\
	=&-\frac{\mathrm{d}}{\mathrm{d}t}\brk{\rev{6}\int|u|^{6}\dx}+\Re\int |u|^4u\wb v_t\dx,
}
then we have
\EQ{
	\frac{\mathrm{d}}{\mathrm{d}t}E(t) = \Re\int |u|^4u\wb v_t\dx=-\Re i\int |u|^4u\De \wb v\dx.
}
This finishes the proof of this lemma.
\end{proof}

Now we continue to prove the proposition. We first consider the mass bound in \eqref{eq:bound-w-bootstrap}. By Lemma \ref{lem:conservation-law} and H\"older's inequality, 
\EQ{
	\sup_{t\in I}M(t)\loe & M(0) + \int_I\absb{ \frac{\mathrm{d}}{\mathrm{d}t}M(t) }\dt\\
	\loe & M(0) + \int_I \absb{\int \wb w\brkb{|u|^4u-|w|^4w}\dx} \dt\\
	\loe & A N_0^{-2s} + C\norm{v}_{L_t^2 L_x^\I} \norm{w}_{L_{t,x}^4} (\norm{w}_{L_{t,x}^4} + \norm{v}_{L_{t,x}^4})(\norm{w}_{L_t^\I L_x^6}^3 + \norm{v}_{L_t^\I L_x^6}^3).
}
Therefore, combining with \eqref{eq:bound-v-ltinfty}, \eqref{eq:bound-v-l2linfty}, \eqref{eq:bound-w-ltinfty}, and \eqref{eq:bound-w-ltx4}, 
\EQn{\label{esti:mass-estimate}
	\sup_{t\in I}M(t) \loe & A N_0^{-2s} + C(A) N_0^{-\frac12a-s+\frac32-3s+} \loe \frac{3}{2}A N_0^{-2s},
}
where we take $N_0=N_0(A)$ such that $C(A)N_0^{-\frac12a+\frac32-2s+}\loe\frac12 A$. This is allowed since $s>\frac34-\frac14a$.

We then prove the energy bound in \eqref{eq:bound-w-bootstrap}. By Lemma \ref{lem:conservation-law} and H\"older's inequality, 
\EQ{
	\sup_{t\in I}E(t)\loe & E(0) + \int_I\absb{ \frac{\mathrm{d}}{\mathrm{d}t}E(t) }\dt\\
	\loe & E(0) + \int_I \absb{\int |u|^4u\cdot\De \wb v\dx} \dt\\
	\loe & A N_0^{2} + C\norm{\De v}_{L_t^2 L_x^\I} (\norm{w}_{L_{t,x}^4}^2 + \norm{v}_{L_{t,x}^4}^2)(\norm{w}_{L_t^\I L_x^6}^3 + \norm{v}_{L_t^\I L_x^6}^3).
}
Therefore, combining with \eqref{eq:bound-v-ltinfty}, \eqref{eq:bound-v-l2linfty}, \eqref{eq:bound-w-ltinfty}, and \eqref{eq:bound-w-ltx4}, 
\EQn{\label{esti:energy-estimate}
\sup_{t\in I}E(t) \loe & A N_0^{2(1-s)} + C(A) N_0^{2-\frac12a-s+\frac32-3s+} \loe \frac{3}{2}A N_0^{2(1-s)},
}
where we still need to take $N_0=N_0(A)$ such that $C(A)N_0^{-\frac12a+\frac32-2s+}\loe\frac12 A$. Therefore,  \eqref{esti:mass-estimate} and \eqref{esti:energy-estimate} gives \eqref{eq:bound-w-bootstrap}. This finishes the proof of this proposition.
\end{proof}

\subsection{Perturbations}\label{sec:perturbation}
Now, we consider the original energy critical equation:
\EQn{\label{eq:nls-w-original}
	\left\{ \aligned
	&i\pd_t \wt w + \De \wt w =  |\wt w|^4\wt w, \\
	& \wt w(0,x) = \wt w_0,
	\endaligned
	\right.
}
where $\wt w(t,x):\R\times \R^3\ra \C$. Let $g(t,x) := w(t,x)-\wt w(t,x)$. Then, the equation for $g$ is
\EQn{
	\label{eq:nls-w-difference}
	\left\{ \aligned
	&i\pd_t g + \De g =  F(g+v,\wt w), \\
	& g(0,x) = w_0 - \wt w_0.
	\endaligned
	\right.
}
Here we denote that
\EQ{
	F(g,w):=|g+w|^4 (g+w) - |w|^4 w.
}
Recall that
\EQ{
	\norm{w}_{X(I)}:=\norm{\jb{\nabla}w}_{L_t^2 L_x^{6}(I\times \R^3)} + \norm{w}_{L_{t,x}^8(I\times \R^3)} + \norm{w}_{L_t^8 L_x^{12}(I\times \R^3)},
}
and
\EQ{
	\norm{v}_{Y(I)}=& \normb{\jb{\nabla}^{s+\frac12a-}v}_{L_t^2 L_x^\I(I\times \R^3)} + \norm{v}_{L_{t,x}^8(I\times \R^3)} + \norm{v}_{L_{t,x}^4(I\times \R^3)} \\
	&+ \norm{v}_{L_t^8 L_x^{12}(I\times \R^3)}.
}
\begin{lem}\label{lem:short-time}
	Let $a\in\N$, $a> 10$, $\frac34-\frac14a<s<0$, $I\subset \R$, and $0\in I$. Then, there exists $0<\eta_1\ll1$ with the following properties. Let $\wt w\in C\brkb{I;H_x^1(\R^3)}$ be the solution of \eqref{eq:nls-w-original} on $I$, satisfying
	\EQ{
		\norm{w_0-\wt w_0}_{\dot H_x^1(\R^3)}\loe \eta_1\text{, and } \norm{\wt w}_{X(I)}\loe \eta_1.
	}
	For any $0<\eta\loe\eta_1$, suppose that
	\EQ{
		\norm{v}_{Y(I)}\loe \eta,
	}
	then there exists a solution $w\in C\brkb{I; H_x^1(\R^3)}$ of \eqref{eq:nls-w} with initial data $w_0$ such that
	\EQ{
		\norm{w-\wt w}_{L_t^\I  H_x^1(I\times\R^3)} + \norm{w-\wt w}_{X(I)} \loe C_0\brkb{\norm{w_0-\wt w_0}_{ H_x^1(\R^3)} + \eta},
	}
	where $C_0>1$ is an absolute constant independent of $\eta$, $\eta_1$ and $I$.
\end{lem}
\begin{proof}
	Let $g$ be the solution of \eqref{eq:nls-w-difference}, $0=\inf I$, and we restrict the time interval on $I$. Then, we have
	\EQ{
		g=e^{it\De}g(0)-i\int_0^te^{i(t-s)\De}F(g+v,\wt w)\ds.
	}
	Note that we have the pointwise estimate
	\EQ{
	\abs{F(g+v,\wt w)} = & \abs{|g+v+\wt w|^4(g+v+\wt w)-|\wt w|^4\wt w}\\
	\lsm & (|g| + |v|)(|g|^4 + |v|^4 + |\wt w|^4 ),
	}
and
	\EQ{
		\abs{\nabla F(g+v,\wt w)}=&\abs{\nabla \brk{|g+v+\wt w|^4(g+v+\wt w)-|\wt w|^4\wt w}}\\
		\lsm& \brk{\abs{\nabla g} + \abs{\nabla v}}\brk{|g|^4 + |v|^4 + |\wt w|^4} + \abs{\nabla \wt w}\brk{|g|^4 + |v|^4}\\
		& + \abs{\nabla \wt w}\brk{|g| + |v|}|\wt w|^3.
	}
Note that by the definition of $a$ and $s$, we have $s+\frac12 a>1$. Therefore, since $v$ is high-frequency,
\EQ{
\norm{v}_{L_t^2 L_x^\I}+ \norm{\nabla v}_{L_t^2 L_x^\I}\lsm \norm{v}_{Y}.
} 
Then, using the similar argument in Lemma \ref{lem:local},
\EQ{
\norm{F(g+v,\wt w)}_{L_t^1 L_x^2} \lsm & \norm{g}_{L_t^2 L_x^6}(\norm{g}_{L_t^8 L_x^{12}}^4 + \norm{v}_{L_t^8 L_x^{12}}^4 + \norm{\wt w}_{L_t^8 L_x^{12}}^4) \\
& + \norm{v}_{L_t^2 L_x^\I}(\norm{g}_{L_{t,x}^8}^4 + \norm{v}_{L_{t,x}^8}^4 + \norm{\wt w}_{L_{t,x}^8}^4) \\
\lsm & (\norm{g}_{X} + \norm{v}_{Y})(\norm{g}_{X}^4 + \norm{v}_{Y}^4 + \norm{\wt w}_{X}^4),
}
and
\EQ{
\norm{\nabla F(g+v,\wt w)}_{L_t^1 L_x^2} \lsm & \norm{\nabla g}_{L_t^2 L_x^6}(\norm{g}_{L_t^8 L_x^{12}}^4 + \norm{v}_{L_t^8 L_x^{12}}^4 + \norm{\wt w}_{L_t^8 L_x^{12}}^4) \\
& + \norm{\nabla v}_{L_t^2 L_x^\I}(\norm{g}_{L_{t,x}^8}^4 + \norm{v}_{L_{t,x}^8}^4 + \norm{\wt w}_{L_{t,x}^8}^4) \\
& +  \norm{\nabla \wt w}_{L_t^2 L_x^6}(\norm{g}_{L_t^8 L_x^{12}}^4 + \norm{v}_{L_t^8 L_x^{12}}^4) \\
& + \norm{\nabla \wt w}_{L_t^2 L_x^6}(\norm{g}_{L_t^8 L_x^{12}} + \norm{v}_{L_t^8 L_x^{12}})\norm{\wt w}_{L_t^8 L_x^{12}}^3 \\
\lsm & (\norm{g}_{X} + \norm{v}_{Y})(\norm{g}_{X}^4 + \norm{v}_{Y}^4 + \norm{\wt w}_{X}^4) \\
& + \norm{\wt w}_{X}(\norm{g}_{X}^4 + \norm{v}_{Y}^4 + \norm{g}_{X}\norm{\wt w}_{X}^3 + \norm{v}_{Y}\norm{\wt w}_{X}^3).
}
Therefore, by Lemma \ref{lem:strichartz},
	\EQ{
		\norm{g}_{L_t^\I H_x^1 \cap X} \lsm & \norm{g(0)}_{ H_x^1} + \norm{\jb{\nabla} F(g+v,\wt w)}_{L_t^1 L_x^2}  \\
		\lsm & \norm{g(0)}_{ H_x^1} + (\norm{g}_{X} + \norm{v}_{Y})(\norm{g}_{X}^4 + \norm{v}_{Y}^4 + \norm{\wt w}_{X}^4) \\
		& + \norm{\wt w}_{X}(\norm{g}_{X}^4 + \norm{v}_{Y}^4 + \norm{g}_{X}\norm{\wt w}_{X}^3 + \norm{v}_{Y}\norm{\wt w}_{X}^3) \\
		\lsm & \norm{g(0)}_{ H_x^1} + \brkb{\norm{g}_{X} + \eta}\brkb{\norm{g}_{X}^4 + \eta^4 + \eta_1^4} \\
		&+ \eta_1\brkb{\norm{g}_{X}^4 + \eta^4 + \norm{g}_{X}\eta_1^3 + \eta\eta_1^3} \\
		\lsm & \norm{g(0)}_{ H_x^1} + \norm{g}_{X}^5 + \eta_1^4\norm{g}_{X} + \eta\norm{g}_{X}^4 +\eta\eta_1^4 + \eta_1\norm{g}_{X}^4,
	}
	then we have
	\EQ{
		\norm{g}_{L_t^\I H_x^1 \cap S} \lsm & \norm{g(0)}_{ H_x^1} + \brkb{\norm{g}_{X}+\eta_1}\norm{g}_{X}^4 + \eta_1^4 \norm{g}_{X} + \eta_1^4\eta.
	}
	Then, this lemma follows by the standard continuity argument.
\end{proof}
\begin{lem}\label{lem:long-time}
	Suppose that $a\in\N$, $a> 10$, $\frac34-\frac14a<s<0$, $I\subset \R$, $M_0>1$, $I\subset\R$, and $0\in I$. Let $\wt w\in C\brkb{I;H_x^1(\R^3)}$ be the solution of \eqref{eq:nls-w-original} on $I$ with
	\EQ{
		\wt w_0=w_0\text{, and }\norm{\wt w}_{X(I)}\loe M_0.
	}
	Let $0<\eta_1\ll1$ and $C_0>1$ be defined as in Lemma \ref{lem:short-time}. Then, there exists $\eta_2=\eta_2(C_0,M_0,\eta_1)>0$ such that if $v$ satisfies
	\EQ{
		\norm{v}_{Y(I)}\loe \eta_2,
	}
	then there exists a solution $w\in C\brkb{I;\dot H_x^1(\R^3)}$ of \eqref{eq:nls-w} with initial data $w_0$ such that
	\EQn{\label{eq:long-time}
		\norm{w-\wt w}_{L_t^\I  H_x^1(I\times\R^3)} + \norm{w-\wt w}_{X(I)} \loe C(C_0,M_0,\eta_1)\eta_2.
	}
\end{lem}
\begin{proof}
	First, we divide the time interval $I$ as consecutive sub-intervals $I=\cup_{j=1}^J I_j$, such that
	\EQ{
		\half 1\eta_1\loe \norm{\wt w}_{X(I_j)}\loe \eta_1,
	}
	where $\eta_1$ is defined in Lemma \ref{lem:short-time}. Let $t_{j-1}=\inf I_j$ and assume without loss of generality that $t_0=0$. Then, we have
	\EQ{
		J=J(M_0,\eta_1).
	}
	From now on, we set another parameter $\eta_2=\eta_2(C_0,M_0,\eta_1)>0$ such that
	\EQn{\label{esti:long-time-eta2}
		(2C_0)^J\eta_2\loe \eta_1.
	}
	We take $v$ such that
	\EQ{
		\norm{v}_{S(I)}\loe \eta_2.
	}
	
	We start from the first interval $I_1$. In this case, $w(0)-\wt w(0)=0$. Then, applying Lemma \ref{lem:short-time} with $\eta=\eta_2$ on $I=[t_0,t_1]$, we obtain the existence of $w$ on $I_1$, and
	\EQn{\label{esti:long-time-I1}
		\norm{w-\wt w}_{L_t^\I  H_x^1(I_1\times\R^3)\cap X(I_1)} \loe C_0\eta\loe 2C_0\eta_2.
	}
	
	Now, we can start the induction procedure. Our aim is to prove for all $k=1,...,J$,
	\EQn{\label{esti:long-time-Ik}
		\norm{w-\wt w}_{L_t^\I  H_x^1(I_k\times\R^3)\cap X(I_k)} \loe (2C_0)^k\eta_2.
	}
	\eqref{esti:long-time-I1} shows that \eqref{esti:long-time-Ik} holds when $k=1$. We assume that \eqref{esti:long-time-Ik} holds for $k=j<J$, namely
	\EQn{\label{esti:long-time-Ij}
		\norm{w-\wt w}_{L_t^\I  H_x^1(I_j\times\R^3)\cap X(I_j)} \loe (2C_0)^j\eta_2.
	}
	It suffices to prove \eqref{esti:long-time-Ik} holds for $k=j+1$. By \eqref{esti:long-time-eta2} and \eqref{esti:long-time-Ij}, we have
	\EQ{
		\norm{w(t_j)-\wt w(t_j)}_{ H_x^1(\R^3)} \loe (2C_0)^j\eta_2\loe (2C_0)^J\eta_2\loe \eta_1.
	}
	Then, we can apply Lemma \ref{lem:short-time} with $\eta=\eta_2$ on $I_{j+1}$ after translation in $t$ starting from $t_j$, and obtain the existence of $w$ on $I_{j+1}$, and
	\EQ{
		\norm{w-\wt w}_{L_t^\I H_x^1(I_{j+1}\times\R^3)\cap X(I_{j+1})} \loe& C_0\brk{\norm{w(t_j)-\wt w(t_j)}_{ H_x^1(\R^3)} +\eta_2}\\
		\loe & C_0\brk{(2C_0)^j\eta_2 +\eta_2}\loe (2C_0)^{j+1}\eta_2.
	}
	This gives \eqref{esti:long-time-Ik} for $k=j+1$.
	
	Then, we have the existence of $w$ on $I$ and for all $k=1,...,J$,
	\EQ{
		\norm{w-\wt w}_{L_t^\I  H_x^1(I_k\times\R^3)\cap X(I_k)} \loe (2C_0)^k\eta_2.
	}
	Therefore, \eqref{eq:long-time} follows by summation over $k$.
\end{proof}

\subsection{Proof of Proposition \ref{prop:global-derterministic}}\label{sec:global-scattering}
We need to use the following classical result:
\begin{lem}\label{lem:energy-critical-classical}
	Suppose that $\wt w_0 \in H^1(\R^3)$. Then, the equation \eqref{eq:nls-w-original} is globally well-posed and scatters, and the solution $\wt w\in C\brko{\R; H^1(\R^3)}$ satisfies
	\EQ{
		\norm{\wt w}_{X(\R)} \loe C(\norm{\wt w_0}_{ H_x^1(\R^3)}).
	}
\end{lem}
\begin{proof}
	Using the result in \cite{CKSTT08Annals}, we have the global well-posedness and the space-time bound
	\EQ{
		\norm{\wt w}_{L_{t,x}^{10}(\R)} \loe C(\norm{\wt w_0}_{ H_x^1(\R^3)}).
	}
	Then, given $0<\ep\ll 1$, we can split  $\R=\cup_{j=1}^J I_j$ such that
	\EQ{
		\norm{\wt w}_{L_{t,x}^{10}(I_j)} \loe \ep.
	}
	Using the equation \eqref{eq:nls-w-original},
	\EQ{
		\norm{\wt w}_{X(I_j)} \loe C(\norm{\wt w_0}_{ H_x^1(\R^3)}).
	}
	Therefore, the lemma follows by summation over $I_j$.
\end{proof}
\textbf{Proof of global well-posedness:} We first prove the global well-posedness and space-time norm estimate by iterating the perturbation theory. We consider only the forward time interval $[0,+\I)$ case. By Proposition \ref{prop:energy-bound}, we have that if $w\in C([0,T];H_x^1)$ solves \eqref{eq:nls-w} for some $T>0$, then
\EQn{\label{esti:energy-bound}
	\sup_{t\in[0,T]} \norm{w(t)}_{ H^1(\R^3)} \loe C(A,N_0)=:E_0.
}
Combining Lemma \ref{lem:energy-critical-classical}, it holds that for any $t'\in[0,\I)$, there exists a solution $\wt w(t,x) = \wt w^{(t')}(t,x)$ of
\EQ{
	\left\{ \aligned
	&i\pd_t \wt w + \De \wt w =  |\wt w|^4\wt w, \\
	& \wt w(t',x) = w(t',x),
	\endaligned
	\right.
}
such that
\EQn{\label{esti:spacetime-bound-fixedtime}
	\normb{\wt w ^{(t')}}_{L_t^\I H_x^1(\R)} + \normb{\wt w ^{(t')}}_{X(\R)} \loe C(\norm{w(t')}_{\dot H_x^1(\R^3)}).
}
By \eqref{esti:energy-bound} and \eqref{esti:spacetime-bound-fixedtime}, we have that if $w\in C([0,T];H_x^1)$ solves \eqref{eq:nls-w} for some $T>0$, then
\EQn{\label{esti:spacetime-bound}
	\sup_{t'\in[0,T]} \brkb{\normb{\wt w ^{(t')}}_{L_t^\I H_x^1(\R)} + \normb{\wt w ^{(t')}}_{X(\R)}} \loe C(E_0)=:M_0,
}
where $M_0$ depends only on $A$ and $N_0$.

By the assumption in Proposition \ref{prop:global-derterministic}, we have $\norm{v}_{Y([0,\I))}\loe A$. Next, we split $[0,\I)=\cup_{l=1}^L I_l$ such that
\EQ{
	\half 1 \eta_2\loe \norm{v}_{Y(I_l)} \loe \eta_2,
}
where $\eta_2=\eta_2(M_0)$ is defined in Lemma \ref{lem:long-time}. Then, $L$ may depend on $M_0$, $A$ and $\eta_2$. Let $s_{l-1}=\inf I_{l}$ and $s_0=0$. We can start from $I_1$. By \eqref{esti:spacetime-bound}, we have
\EQ{
	\normb{\wt w ^{(s_0)}}_{S(\R)} \loe M_0.
}
Then, we can apply Lemma \ref{lem:long-time} on $I_1$ to obtain the existence of $w\in C([s_0,s_1]; H^1(\R^d))$. Furthermore, by Proposition \ref{prop:energy-bound}, we can get
\EQ{
	\sup_{t\in I_1} \norm{w(t)}_{ H^1(\R^3)} \loe E_0\text{, and }\norm{w}_{X(I_1)}\loe M_0.
}
Particularly for $\wt w^{(s_1)}$, we have $\norm{\wt w^{(s_1)}(s_1)}_{ H^1(\R^3)}=\norm{w(s_1)}_{ H^1(\R^3)} \loe E_0$. Using \eqref{esti:spacetime-bound} again,
\EQ{
	\normb{\wt w ^{(s_1)}}_{X(\R)} \loe M_0.
}
Then, we can apply Lemma \ref{lem:long-time} on $I_2$ after translation in $t$ from the starting point $s_1$. Therefore, we obtain the existence of $w\in C([s_1,s_2]; H^1(\R^d))$, and
\EQ{
	\sup_{t\in I_2} \norm{w(t)}_{ H^1(\R^3)} \loe E_0\text{, and }\norm{w}_{X(I_2)}\loe M_0.
}

Inductively, for all $l=1,2,...,L$, we can obtain that $w\in C\brkb{I_l; H^1(\R^3)}$, and also
\EQ{
	\norm{w}_{X(I_l)}\loe M_0.
}
Therefore, we have $w\in C\brkb{[0,\I); H^1(\R^3)}$, and
\EQ{
\norm{w}_{X([0,\I))}\loe LM_0=C(A,M_0,\eta_2)=C(A).
}

\textbf{Proof of scattering:} Next, we prove the scattering statement. We only consider the $t\ra+\I$ case, and it suffices to prove that
\EQn{\label{eq:scattering-space-time}
\normb{\jb{\nabla} \int_0^\I e^{-is\De}(|u|^4u)\dx}_{L_x^2}\loe C(A).
}
In fact, since the global well-posedness already holds, we do not care the explicit expression of $A$. Now, all the space-time norms are taken over $[0,+\I)\times \R^3$. From previous argument,  
\EQ{
\norm{w}_{X([0,+\I))}\loe C(A).
}
Recall also that 
\EQ{
\norm{v}_{Y(\R)}\loe A.
}
Now, we can prove \eqref{eq:scattering-space-time} using the argument in Lemma \ref{lem:local},
\EQ{
\text{L.H.S. of }\eqref{eq:scattering-space-time} \lsm & \normb{ \int_0^\I e^{-is\De}( |u|^4u)\ds}_{L_x^2} + \normb{ \int_0^\I e^{-is\De}(\nabla wu^4)\ds}_{L_x^2} \\
& + \normb{ \int_0^\I e^{-is\De}(\nabla v u^4)\ds}_{L_x^2}\\
\lsm & \norm{|u|^4u}_{L_t^1 L_x^2} + \norm{\nabla wu^4}_{L_t^1 L_x^2} + \norm{\nabla vu^4}_{L_t^1 L_x^2}\\
\lsm & \norm{\jb{\nabla}w}_{L_t^2 L_x^6}\norm{u}_{L_t^8 L_x^{12}}^4 + \norm{\jb{\nabla}v}_{L_t^2 L_x^\I}\norm{u}_{L_{t,x}^8}^4\\
\loe & C(A).
}
This finishes the proof of scattering statement.

\section{Global well-posedness and scattering in 4D case}\label{sec:gwp-4d}

\vskip .5cm

Now, we give the proof of Theorem \ref{thm:global} in 4D case. The argument is parallel to the 3D case, so we only give a sketch of the proof and highlight the different part.
\subsection{Reduction to the deterministic problem}\label{sec:reduction-global-4d}
Suppose that $u=v+w$ with $u_0=v_0+w_0$, $v=e^{it\De}v_0$, and $w$ satisfying
\EQn{
	\label{eq:nls-w-4d}
	\left\{ \aligned
	&i\pd_t w + \De w = |u|^2 u, \\
	& w(0,x) = w_0(x).
	\endaligned
	\right.
}
Recall that
\EQ{
\norm{v}_{Y(I)}:=&\normb{\jb{\nabla}^{s+a-}v}_{L_t^2 L_x^\I(I\times \R^4)} + \norm{v}_{L_t^4 L_x^8(I\times \R^4)} + \norm{v}_{L_t^6 L_x^3(I\times \R^4)} \\
& + \normb{\jb{\nabla}^{-\frac14}v}_{L_{t,x}^4(I\times \R^4)},
}
and
\EQ{
\norm{v}_{Z(I)}:=  \norm{v}_{L_t^\I H_x^s(\R\times\R^4)} 
+ \normb{\jb{\nabla}^{s+2 a-}v}_{L_t^\I L_x^{\I}(\R\times\R^4)}.
}
We define the energy as
\EQn{\label{defn:energy-w-4d}
	E(t):= \frac12\int_{\R^4}|\nabla w(t,x)|^2 \dx + \frac{1}{4}\int_{\R^4}|u(t,x)|^{4}\dx,
}
and the mass as
\EQn{\label{defn:mass-w-4d}
	M(t):=\int_{\R^4}|w(t,x)|^2\dx.
}
\begin{prop}\label{prop:global-derterministic-4d}
	Let $a\in\N$, $a> 10$, $\frac12-\frac12a<s<0$, and $A>0$. Then, there exists $N_0=N_0(A)\gg1$ such that the following properties hold.  Let $u_0\in L_x^2(\R^4)$, $v_0$ satisfy that $\supp\wh v_0\subset \fbrk{\xi\in\R^4:|\xi|\goe \frac12 N_0}$,  and $w_0=u_0-v_0$. Moreover, let $v=e^{it\Delta}v_0$ and $w=u-v$. Suppose that  $v\in Y\cap Z(\R)$, $w_0\in H^1(\R^4)$ such that
	\EQ{
		\norm{u_0}_{L_x^2}+\norm{v}_{Y\cap Z(\R)}\loe A\text{, and }E(w_0) \loe A N_0^{2}.
	}
	Then, there exists a solution $u$ of \eqref{eq:nls-3D} on $\R$ with $w\in C(\R;H_x^{1}(\R^4))$. Furthermore, there exists $u_{\pm}\in H_x^1(\R^4)$ such that
	\EQ{
		\lim_{t\ra\pm\I}\norm{u(t)-v(t)-e^{it\De}u_{\pm}}_{H_x^1(\R^4)}=0.
	}
\end{prop}
We will give the proof of Proposition \ref{prop:global-derterministic-4d} in Sections \ref{sec:local-4d}-\ref{sec:global-scattering-4d}. Now, we can prove Theorem \ref{thm:global} in 4D case assuming that Proposition \ref{prop:global-derterministic-4d} holds, using the same argument in 3D case. 
\subsection{Local theory}\label{sec:local-4d}
We define the space $X(I)$ as
\EQ{
	\norm{w}_{X(I)}:=\norm{\jb{\nabla}w}_{L_t^2 L_x^{4}(I\times \R^4)} + \norm{w}_{L_t^4 L_x^8(I\times \R^4)} + \norm{w}_{L_{t,x}^4(I\times \R^4)}.
}
\begin{lem}[Local well-posedness]\label{lem:local-4d}
	Let $a\in\N$, $a> 10$, $\frac12-\frac12a<s<1$, $v\in Y\cap Z(\R)$, and $w_0\in H_x^1$. Then, there exists some $T>0$ depending on $a$, $w_0$, and $v$ such that there exists a unique solution $w$ of \eqref{eq:nls-w} in some $0$-neighbourhood of 
	\EQ{
		C([0,T];H_x^1(\R^4)) \cap X([0,T]).
	}
\end{lem}
\begin{proof}
First, we make the choices of some parameters:
\enu{
	\item Let $C_0$ be the constant such that
	\EQ{
		\norm{e^{it\De}w_0}_{L_t^\I(\R; H_x^1)\cap X(\R)}\loe C_0\norm{w_0}_{H_x^1}.
	}
	\item 
	Define
	\EQ{
		R:=\max\fbrk{C_0\norm{w_0}_{H_x^1},1}.
	}
	\item 
	Let $\de>0$ be some small constant such that $C\de^2 R^2\loe \frac12$.
	\item 
	Let $T>0$ satisfy the smallness condition
	\EQ{
		\norm{e^{it\De}w_0}_{X([0,T])} + \norm{v}_{Y([0,T])}\loe\de R.
	}
}
Now, we define the working space as
\EQ{
	B_{R,\de,T}:=\fbrk{w\in C([0,T];H_x^1):\norm{w}_{L_t^\I H_x^1([0,T]\times\R^4)} + \de^{-1}\norm{w}_{X([0,T])}\loe 4 R},
}
equipped with the norm
\EQ{
	\norm{w}_{B_{R,\de,T}}:= \norm{w}_{L_t^\I H_x^1([0,T]\times\R^4)} + \de^{-1}\norm{w}_{X([0,T])}.
}
Take the solution map as
\EQ{
	\Phi_{w_0,v}(w):=e^{it\De}w_0-i\int_0^t e^{i(t-s)\De}(|u|^2u)\ds.
}
Then, it suffices to prove that $\Phi_{w_0,v}$ is a contraction mapping on $B_{R,\de,T}$.

Now, we only prove that for any $w\in B_{R,\de,T}$, $\Phi_{w_0,v}(w)\in B_{R,\de,T}$. By Lemma \ref{lem:strichartz} and H\"older's inequality,
\EQ{
	\norm{\Phi_{w_0,v}(w)}_{L_t^\I H_x^1} \loe & C_0\norm{w_0}_{H_x^1} + C\norm{\nabla (|u|^2u)}_{L_t^1 L_x^2} + C\norm{|u|^2u}_{L_t^1 L_x^2} \\
	\loe & R + C\norm{\nabla w}_{L_t^2 L_x^4}\norm{u}_{L_t^4 L_x^{8}}^2 + C\norm{\nabla v}_{L_t^2 L_x^\I} \norm{u}_{L_{t,x}^4}^2 \\
	&+ C\norm{u}_{L_t^2 L_x^4}\norm{u}_{L_t^4 L_x^{8}}^2\\
	\loe & R + C\de^3 R^3 \loe 2R.
}

Similar as above, we have
\EQ{
	\norm{\Phi_{w_0,v}(w)}_{X} \loe & \norm{e^{it\De}w_0}_{H_x^1} + C\norm{\nabla (|u|^2u)}_{L_t^1 L_x^2} + C\norm{|u|^2u}_{L_t^1 L_x^2} \\
	\loe & \de R + C\de^3 R^3 \loe2\de R.
}
Therefore, we have
\EQ{
	\norm{\Phi_{w_0,v}(w)}_{B_{R,\de,T}} \loe \norm{\Phi_{w_0,v}(w)}_{L_t^\I H_x^1} + \de^{-1}\norm{\Phi_{w_0,v}(w)}_{X}\loe 4R.
}
This shows that $\Phi_{w_0,v}$ maps $B_{R,\de,T}$ into itself. Since we already establish the non-linear estimates, the contraction mapping statement follows by similar argument in 3D case. 
\end{proof}

\subsection{Modified Interaction Morawetz}\label{sec:space-time-4d}

\begin{lem}[Modified Interaction Morawetz]\label{lem:inter-mora-4d}
	Given $T>0$. Let $w\in C([0,T];H_x^1)$ be the solution of perturbation equation \eqref{eq:nls-w}. Then, we have
	\EQn{\label{eq:inter-mora-4d}
		\normb{|\nabla|^{-\frac14}w}_{L_{t,x}^4}^4\lsm & \norm{w}_{L_t^\I L_x^2}^2 \norm{w}_{L_t^\I \dot H_x^{\frac12}}^2 \\
		&+ \norm{w}_{L_t^\I H_x^{\frac12}} \normb{|\nabla|^{-\frac14}w}_{L_{t,x}^4}^2\brkb{\norm{v}_{L_t^2 L_x^\I}\norm{w}_{L_t^\I \dot H_x^{\frac12}}^2 + \norm{\nabla v}_{L_t^2 L_x^\I}\norm{w}_{L_t^\I L_x^2}^2} \\
		& + \norm{v}_{L_t^2 L_x^\I}\norm{w}_{L_t^\I \dot H_x^{\frac12}}^2\brkb{\norm{w}_{L_t^\I L_x^2} \norm{v}_{L_{t,x}^4}^2 + \norm{v}_{L_t^6 L_x^3}^3},
	}
	where all the space-time norms are taken over $[0,T]\times \R^4$.
\end{lem}
\begin{proof}
	Recall that $w$ satisfies
	\EQ{
		i\pd_t w + \De w=|w|^2w+e,
	}
	where we denote $e:=|u|^2u-|w|^2w$.
	Denote that 
	$$
	m(t,x)=\frac12|w(t,x)|^2;\quad 
	p(t,x)=\frac12\im \brko{\wb w(t,x)\nabla w(t,x)}.
	$$
	Then, we have
	\EQn{\label{eq:local-mass-flow-4d}
		\pd_tm=-2\nabla\cdot p+ \im\brk{e\bar w},
	}
	and
	\EQn{\label{eq:local-momentum-flow-4d}
		\pd_tp = & -\re \nabla\cdot \brko{\nabla \wb w \nabla w}  -\frac14 \nabla \brkb{|w|^4} + \half 1 \nabla \De m
		+ \re \brk{\wb e\nabla w}-\frac12\re\nabla\brk{\wb w e}.
	}
	Let 
	\EQ{
		M(t):= \int\!\!\int_{\R^{4+4}} \frac{x-y}{|x-y|}\cdot p(t,x)\> m(t,y)\dx\dy,
	}
	then by \eqref{eq:local-mass-flow-4d} and \eqref{eq:local-momentum-flow-4d}, we have the interaction Morawetz identity
	\begin{subequations}\label{5.7-a-f}
		\EQnn{
			\pd_t M(t)
			= &\iint_{\R^{4+4}} \frac{x-y}{|x-y|}\cdot \partial_t p(t,x)\> m(t,y)\dx\dy\nonumber\\
			& \quad + \iint_{\R^{4+4}} \frac{x-y}{|x-y|}\cdot p(t,x)\> \partial_t m(t,y)\dx\dy\nonumber\\
			= &\iint_{\R^{4+4}} \frac{x-y}{|x-y|}\cdot \Big( -\re \nabla\cdot \brko{\nabla \wb w \nabla w}  -\frac14 \nabla \brkb{|w|^4}\Big)(t,x)\> m(t,y)\dx\dy\nonumber\\
			& \quad -2\iint_{\R^{4+4}} \frac{x-y}{|x-y|}\cdot p(t,x)\> \nabla\cdot p(t,y)\dx\dy\nonumber\\
			&\quad +\frac12\iint_{\R^{4+4}} \frac{x-y}{|x-y|}\cdot \nabla \De m(t,x)\> m(t,y)\dx\dy\label{esti:inter-mora-angular-3-4d}\\
			&\quad +\iint_{\R^{4+4}} \frac{x-y}{|x-y|}\cdot  p(t,x)\>  \im\brk{e\bar w}(t,y)\dx\dy\label{esti:inter-mora-remainder-1-4d}\\
			& \quad + \iint_{\R^{4+4}} \frac{x-y}{|x-y|}\cdot \re \brk{\wb e\nabla w}(t,x)\> m(t,y)\dx\dy\label{esti:inter-mora-remainder-2-4d}\\
			& \quad + \iint_{\R^{4+4}} \frac1{|x-y|}\cdot \re \brk{\wb e w}(t,x)\> m(t,y)\dx\dy.\label{esti:inter-mora-remainder-3-4d}
		}
	\end{subequations}
	Different from the 3D case, we have that 
	\EQ{
		\eqref{esti:inter-mora-angular-3-4d}\gsm \normb{|\nabla|^{-\frac14}w(t)}_{L_x^4}^4.
	}
	Then treating similar as in the proof of Lemma \ref{lem:inter-mora}, we obtain that 
	\EQ{
		\normb{|\nabla|^{-\frac14}w}_{L_{t,x}^4}^4\lsm \norm{w}_{L_t^\I L_x^2}^2 \norm{w}_{L_t^\I \dot H_x^{\frac12}}^2+ \int_0^T |\eqref{esti:inter-mora-remainder-1-4d}| + |\eqref{esti:inter-mora-remainder-2-4d}| +|\eqref{esti:inter-mora-remainder-3-4d}|\dt.
	}
	Next, we estimate the terms containing \eqref{esti:inter-mora-remainder-1-4d}, \eqref{esti:inter-mora-remainder-2-4d}, and \eqref{esti:inter-mora-remainder-3-4d}. We first consider \eqref{esti:inter-mora-remainder-1-4d}. By H\"older's inequality,
	\EQ{
		\int_0^T |\eqref{esti:inter-mora-remainder-1-4d}| \dt \lsm & \int_0^T \absb{\iint_{\R^{4+4}} \frac{x-y}{|x-y|}\cdot  p(t,x)\>  \im\brk{e\bar w}(t,y)\dx\dy} \dt \\
		\lsm & \int_0^T \absb{\int_{\R^{4}}   \im\brk{e\bar w}(t,y)\dy} \dt \sup_{y,t}\absb{\int_{\R^4}\frac{x-y}{|x-y|}\cdot  p(t,x)\dx} \\
		\lsm & \norm{v}_{L_t^2 L_x^\I}\norm{w}_{L_t^6 L_x^3}\brkb{\norm{w}_{L_t^6 L_x^3}^2+ \norm{v}_{L_t^6 L_x^3}^2} \norm{w}_{L_t^\I \dot H_x^{\frac12}}^2.
		}
Note that by Gagliardo-Nirenberg's inequality, we have that 	
\EQn{\label{GN-inequality}
\norm{w}_{L_x^3(\R^4)}^3\lsm \norm{w}_{H_x^{\frac12}(\R^4)}\cdot\normb{|\nabla|^{-\frac14}w}_{L_x^4(\R^4)}^2.
}
This gives  that 	
		
		\EQ{
		\int_0^T |\eqref{esti:inter-mora-remainder-1-4d}| \dt\lsm & \norm{v}_{L_t^2 L_x^\I}\norm{w}_{L_t^\I H_x^{\frac12}} \normb{|\nabla|^{-\frac14}w}_{L_{t,x}^4}^2\norm{w}_{L_t^\I \dot H_x^{\frac12}}^2 \\
		& + \norm{v}_{L_t^2 L_x^\I}\norm{v}_{L_t^6 L_x^3}^3\norm{w}_{L_t^\I \dot H_x^{\frac12}}^2.
	}
	
	We then consider \eqref{esti:inter-mora-remainder-2-4d}, where we need to modify the Morawetz estimate.
	\EQ{
		\int_0^T |\eqref{esti:inter-mora-remainder-2-4d}| \dt \lsm & \int_0^T \absb{\iint_{\R^{4+4}} \frac{x-y}{|x-y|}\cdot \re \brk{(|u|^2u-|w|^2w)\nabla \wb w}(t,x)\> m(t,y)\dx\dy} \dt.
	}
Similar as in the proof of \eqref{Mor-Mainterm}, we have that by integration-by-parts,
	\EQ{
		&\int_{\R^4} \frac{x-y}{|x-y|}\cdot \re \brk{(|u|^2u-|w|^2w)\nabla \wb w}(t,x) \dx \\
		= & -\frac34\int_{\R^4}  \frac{1}{|x-y|}\brk{|u(t,x)|^4-|w(t,x)|^4}\dx - \int_{\R^4}  \frac{x-y}{|x-y|}\cdot \re \brk{|u|^2u\nabla \wb v}(t,x)\dx.
	}
	Therefore, by Lemma \ref{lem:hardy} and \eqref{GN-inequality}, we have
	\EQ{
		\int_0^T |\eqref{esti:inter-mora-remainder-2-4d}| \dt \lsm & \int_0^T \absb{\iint_{\R^{4+4}} \frac{1}{|x-y|}\brk{|u(t,x)|^4-|w(t,x)|^4}\> m(t,y)\dx\dy} \dt \\
		& + \int_0^T \absb{\iint_{\R^{4+4}} \frac{x-y}{|x-y|}\cdot \re \brk{|u|^2u\nabla \wb v}(t,x)\> m(t,y)\dx\dy} \dt\\
		\lsm & \int_0^T \absb{\int_{\R^{4}} \brk{|u(t,x)|^4-|w(t,x)|^4}\dx} \sup_{x}\normb{\frac{1}{|x-\cdot|^{\frac12}}w(t,\cdot)}_{L_y^2}^2 \dt \\
		& + \int_0^T \int_{\R^{4}}  \absb{|u|^2u\nabla \wb v(t,x)}\dx\norm{w(t)}_{L_y^2}^2 \dt\\
		\lsm & \norm{v}_{L_t^2 L_x^\I}\big( \norm{w}_{L_t^6 L_x^3}^3+ \norm{v}_{L_t^6 L_x^3}^3\big)\norm{w}_{L_t^\I \dot H_x^{\frac12}}^2 \\
		& + \norm{\nabla v}_{L_t^2 L_x^\I} \big( \norm{w}_{L_t^6 L_x^3}^3+ \norm{v}_{L_t^6 L_x^3}^3\big) \norm{w}_{L_t^\I L_x^2}^2 \\
		\lsm & \norm{v}_{L_t^2 L_x^\I}\norm{w}_{L_t^\I H_x^{\frac12}} \normb{|\nabla|^{-\frac14}w}_{L_{t,x}^4}^2\norm{w}_{L_t^\I \dot H_x^{\frac12}}^2 \\
		& + \norm{v}_{L_t^2 L_x^\I} \norm{v}_{L_t^6 L_x^3}^3\norm{w}_{L_t^\I \dot H_x^{\frac12}}^2 \\
		& + \norm{\nabla v}_{L_t^2 L_x^\I}\norm{w}_{L_t^\I H_x^{\frac12}} \normb{|\nabla|^{-\frac14}w}_{L_{t,x}^4}^2 \norm{w}_{L_t^\I L_x^2}^2 \\
		& + \norm{\nabla v}_{L_t^2 L_x^\I} \norm{v}_{L_t^6 L_x^3}^3 \norm{w}_{L_t^\I L_x^2}^2.
	}
	
	Now, we consider the term  \eqref{esti:inter-mora-remainder-3-4d}. By H\"older's inequality and \eqref{GN-inequality},
	\EQ{
		\int_0^T |\eqref{esti:inter-mora-remainder-3-4d}| \dt \lsm & \int_0^T \absb{\iint_{\R^{4+4}} \frac{1}{|x-y|} e\wb w(t,x)\> m(t,y)\dx\dy} \dt \\
		\lsm & \int_0^T \int_{\R^{4}} \abs{e\wb w(t,x)}\dx \sup_{x}\normb{\frac{1}{|x-\cdot|^{\frac12}}w(t,\cdot)}_{L_y^2}^2 \dt \\
		\lsm & \norm{v}_{L_t^2 L_x^\I}\norm{w}_{L_t^\I H_x^{\frac12}} \normb{|\nabla|^{-\frac14}w}_{L_{t,x}^4}^2\norm{w}_{L_t^\I \dot H_x^{\frac12}}^2 \\
		& + \norm{v}_{L_t^2 L_x^\I}\norm{w}_{L_t^\I L_x^2} \norm{v}_{L_{t,x}^4}^2\norm{w}_{L_t^\I \dot H_x^{\frac12}}^2.
	}
	
	Combining the three findings on \eqref{5.7-a-f}, we get 
	\EQ{
		&\int_0^T|\eqref{esti:inter-mora-remainder-1-4d}| + |\eqref{esti:inter-mora-remainder-2-4d}| +|\eqref{esti:inter-mora-remainder-3-4d}|\dt\\
		\lsm & \norm{v}_{L_t^2 L_x^\I}\norm{w}_{L_t^\I H_x^{\frac12}} \normb{|\nabla|^{-\frac14}w}_{L_{t,x}^4}^2\norm{w}_{L_t^\I \dot H_x^{\frac12}}^2 \\
		& + \norm{v}_{L_t^2 L_x^\I}\norm{w}_{L_t^\I L_x^2} \norm{v}_{L_{t,x}^4}^2\norm{w}_{L_t^\I \dot H_x^{\frac12}}^2 \\
		& + \norm{v}_{L_t^2 L_x^\I} \norm{v}_{L_t^6 L_x^3}^3\norm{w}_{L_t^\I \dot H_x^{\frac12}}^2 \\
		& + \norm{\nabla v}_{L_t^2 L_x^\I}\norm{w}_{L_t^\I H_x^{\frac12}} \normb{|\nabla|^{-\frac14}w}_{L_{t,x}^4}^2 \norm{w}_{L_t^\I L_x^2}^2\\
		\lsm & \norm{w}_{L_t^\I H_x^{\frac12}} \normb{|\nabla|^{-\frac14}w}_{L_{t,x}^4}^2\brkb{\norm{v}_{L_t^2 L_x^\I}\norm{w}_{L_t^\I \dot H_x^{\frac12}}^2 + \norm{\nabla v}_{L_t^2 L_x^\I}\norm{w}_{L_t^\I L_x^2}^2} \\
		& + \norm{v}_{L_t^2 L_x^\I}\norm{w}_{L_t^\I \dot H_x^{\frac12}}^2\brkb{\norm{w}_{L_t^\I L_x^2} \norm{v}_{L_{t,x}^4}^2 + \norm{v}_{L_t^6 L_x^3}^3}.
	}
Therefore, we have
	\EQ{
		\norm{w}_{L_{t,x}^4}^4\lsm & \norm{w}_{L_t^\I L_x^2}^2 \norm{w}_{L_t^\I \dot H_x^{\frac12}}^2 \\
		&+ \norm{w}_{L_t^\I H_x^{\frac12}} \normb{|\nabla|^{-\frac14}w}_{L_{t,x}^4}^2\brkb{\norm{v}_{L_t^2 L_x^\I}\norm{w}_{L_t^\I \dot H_x^{\frac12}}^2 + \norm{\nabla v}_{L_t^2 L_x^\I}\norm{w}_{L_t^\I L_x^2}^2} \\
		& + \norm{v}_{L_t^2 L_x^\I}\norm{w}_{L_t^\I \dot H_x^{\frac12}}^2\brkb{\norm{w}_{L_t^\I L_x^2} \norm{v}_{L_{t,x}^4}^2 + \norm{v}_{L_t^6 L_x^3}^3}.
	}
	This completes the proof of this lemma.
\end{proof}
\subsection{Almost conservation law}\label{sec:energy-bound-4d}
 
\begin{prop}\label{prop:energy-bound-4d}
	Let $a\in\N$, $a> 10$, $\frac12-\frac12a<s<0$, $A>0$, $v=e^{it\De}v_0\in  Y\cap Z(\R)$ and $w$ be the solution of \eqref{eq:nls-w}. Take some $T>0$ such that $w\in C([0,T];H_x^1)$. Then, there exists $N_0=N_0(A)\gg1$ with the following properties. Assume that $\wh v_0$ is supported on $\fbrk{\xi\in\R^4:|\xi|\goe \frac12 N_0}$,
	\EQ{
		\norm{u_0}_{H_x^s}+\norm{v}_{Y\cap Z(\R)}\loe A\text{, }M(0)\loe AN_0^{-2s}\text{, and }E(0) \loe A N_0^{2(1-s)}.
	}
	Then, we have
	\EQn{\label{eq:energy-bound-4d}
		\sup_{t\in[0,T]}M(t)\loe 2A N_0^{-2s}\text{, and }\sup_{t\in[0,T]}E(t)\loe 2A N_0^{2(1-s)}.
	}
\end{prop}
\begin{proof}
Let $N_0=N_0(A)$ that will be defined later. We implement a bootstrap procedure on $I\subset [0,T]$: assume an a priori bound 
\EQn{\label{eq:bound-w-hypothesis-4d}
	\sup_{t\in I}M(t)\loe 2A N_0^{-2s}\text{, and }\sup_{t\in I} E(t)\loe 2AN_0^{2(1-s)},
} 
then it suffices to prove that
\EQn{\label{eq:bound-w-bootstrap-4d}
	\sup_{t\in I}M(t)\loe \frac32A N_0^{-2s}\text{, and }\sup_{t\in I} E(t)\loe \frac32AN_0^{2(1-s)}.
}
From now on, all the space-time norms are taken over $I\times\R^4$. 

To start with, we collect useful estimates on $I$. Now, we use the notation $C=C(A)$ for short, and  the implicit constants in ``$\lesssim$'' depend on $A$. By interpolation, we have
\EQn{\label{eq:bound-v-ltinfty-4d}
	\norm{v}_{L_t^\I L_x^6} + \normb{\jb{\nabla}^{-\frac14}v}_{L_{t,x}^4} + \norm{v}_{L_t^6 L_x^3}\lsm\norm{v}_{Y\cap Z}\lsm1. 
}
By the frequency support of $v$, we have for any $0\loe l<a+s$,
\EQn{\label{eq:bound-v-l2linfty-4d}
\normb{|\nabla|^lv_N}_{l_N^2L_t^2 L_x^\I}\lsm N_0^{l-a-s+}\norm{v}_{ Y}\lsm N_0^{l-a-s+}\lsm 1.
} 
Note that we assume $a>10$ and $s>\frac12-\frac12a$, then $a + s > 2 $. Therefore, this guarantees that $\norm{\De v}_{L_t^2 L_x^\I} \lsm N_0^{2-a-s+}$ is allowed.
By bootstrap hypothesis \eqref{eq:bound-w-hypothesis-4d},
\EQn{\label{eq:bound-w-h1-4d}
	\norm{w}_{L_t^\I L_x^2}\lsm N_0^{-s}\text{, and } \norm{w}_{L_t^\I \dot H_x^1}\lsm N_0^{1-s}.
}
Then, by interpolation and \eqref{eq:bound-w-h1-4d}, we have for any $0\loe l \loe 1$, 
\EQn{\label{eq:bound-w-ltinfty-4d}
	\norm{w}_{L_t^\I \dot H_x^l}\lsm N_0^{l-s}.
}
Furthermore, by Lemma \ref{lem:inter-mora-4d}, \eqref{eq:bound-w-ltinfty-4d}, and \eqref{eq:bound-v-l2linfty-4d},
\EQ{
\normb{|\nabla|^{-\frac14}w}_{L_{t,x}^4}^4\lsm & \norm{w}_{L_t^\I L_x^2}^2 \norm{w}_{L_t^\I \dot H_x^{\frac12}}^2 \\
&+ \norm{w}_{L_t^\I H_x^{\frac12}} \normb{|\nabla|^{-\frac14}w}_{L_{t,x}^4}^2\brkb{\norm{v}_{L_t^2 L_x^\I}\norm{w}_{L_t^\I \dot H_x^{\frac12}}^2 + \norm{\nabla v}_{L_t^2 L_x^\I}\norm{w}_{L_t^\I L_x^2}^2} \\
& + \norm{v}_{L_t^2 L_x^\I}\norm{w}_{L_t^\I \dot H_x^{\frac12}}^2\brkb{\norm{w}_{L_t^\I L_x^2} \norm{v}_{L_{t,x}^4}^2 + \norm{v}_{L_t^6 L_x^3}^3} \\
\lsm & N_0^{1-4s} + N_0^{\frac12-s+}\normb{|\nabla|^{-\frac14}w}_{L_{t,x}^4}^2\brkb{N_0^{-a-s+}N_0^{1-2s} +N_0^{1-a-s+} N_0^{-2s}}\\
& + N_0^{1-2s} \brkb{N_0^{-s}+1}\\
\lsm & N_0^{1-4s} + N_0^{\frac32-a-4s+}\normb{|\nabla|^{-\frac14}w}_{L_{t,x}^4}^2.
}
Since $s>\frac12-\frac12 a$, by Young's inequality, we have
\EQn{\label{eq:bound-w-ltx4-4d}
	\norm{w}_{L_{t,x}^4}^4\lsm N_0^{1-4s}+ N_0^{\frac32-a-4s+} + N_0^{3-2a-8s+} \lsm N_0^{1-4s}.
}

Now, we are prepared to give the proof of \eqref{eq:bound-w-bootstrap-4d}.
We first consider the mass bound in \eqref{eq:bound-w-bootstrap-4d}. Note that similarly to Lemma \ref{lem:conservation-law},
\EQ{
		\absb{\frac{\mathrm{d}}{\mathrm{d}t}M(t)} \loe 2\absb{\int_{\R^4} \brkb{|u|^2u-|w|^2w} \wb w\dx}.
	}
Then by   H\"older's inequality and \eqref{GN-inequality}, we have that 
\EQ{
	\sup_{t\in I}M(t)\loe & M(0) + \int_I\absb{ \frac{\mathrm{d}}{\mathrm{d}t}M(t) }\dt\\
	\loe & M(0) + \int_I \absb{\int \wb w\brkb{|u|^2u-|w|^2w}\dx} \dt \\
	\loe & AN_0^{-2s} + C\norm{v}_{L_t^2 L_x^\I} (\norm{w}_{L_t^6 L_x^3}^3 + \norm{v}_{L_t^6 L_x^3}^3)\\
	\loe & A N_0^{2} + C \norm{v}_{L_t^2 L_x^\I} \brkb{ \norm{w}_{L_t^\I H_x^{\frac12}} \normb{|\nabla|^{-\frac14}w}_{L_{t,x}^4}^2+ \norm{v}_{L_t^2 L_x^\I}\norm{v}_{L_t^\I L_x^2}^2 }.
}
Therefore, combining with \eqref{eq:bound-v-ltinfty-4d}, \eqref{eq:bound-v-l2linfty-4d}, \eqref{eq:bound-w-ltinfty-4d}, and \eqref{eq:bound-w-ltx4-4d}, 
\EQn{\label{esti:mass-estimate-4d}
	\sup_{t\in I}M(t) \loe & A N_0^{-2s} + C(A) N_0^{-a-s + 1-3s+} \loe \frac{3}{2}A N_0^{-2s},
}
where we take $N_0=N_0(A)$ such that $C(A)N_0^{-a+1-2s+}\loe\frac12 A$. This is allowed since $s>\frac12-\frac12a$.

We then consider the energy bound in \eqref{eq:bound-w-bootstrap-4d}. Note that similarly to Lemma \ref{lem:conservation-law},
\EQ{
\absb{\frac{\mathrm{d}}{\mathrm{d}t}E(t)} \loe \absb{\int_{\R^4} |u|^2u\De \wb v\dx},
}
then by H\"older's inequality and \eqref{GN-inequality},
\EQ{
	\sup_{t\in I}E(t)\loe & E(0) + \int_I\absb{ \frac{\mathrm{d}}{\mathrm{d}t}E(t) }\dt\\
	\loe & E(0) + \int_I \absb{\int |u|^2u\cdot\De \wb v\dx} \dt\\
	\loe & A N_0^{2} + C\norm{\De v}_{L_t^2 L_x^\I} (\norm{w}_{L_t^6 L_x^3}^3 + \norm{v}_{L_t^6 L_x^3}^3)\\
	\loe & A N_0^{2} + C \norm{\De v}_{L_t^2 L_x^\I} \brkb{ \norm{w}_{L_t^\I H_x^{\frac12}} \normb{|\nabla|^{-\frac14}w}_{L_{t,x}^4}^2+ \norm{v}_{L_t^2 L_x^\I}\norm{v}_{L_t^\I L_x^2}^2 }.
}
Therefore, combining with \eqref{eq:bound-v-ltinfty-4d}, \eqref{eq:bound-v-l2linfty-4d}, \eqref{eq:bound-w-ltinfty-4d}, and \eqref{eq:bound-w-ltx4-4d}, 
\EQn{\label{esti:energy-estimate-4d}
	\sup_{t\in I}E(t) \loe & A N_0^{2(1-s)} + C(A) N_0^{2-a-s+1-3s+} \loe \frac{3}{2}A N_0^{2(1-s)},
}
where we still need to take $N_0=N_0(A)$ such that $C(A)N_0^{-a+1-2s+}\loe\frac12 A$. Therefore,  \eqref{esti:mass-estimate-4d} and \eqref{esti:energy-estimate-4d} gives \eqref{eq:bound-w-bootstrap-4d}. This finishes the proof of this proposition.
\end{proof}

\subsection{Perturbations}\label{sec:perturbation-4d}
Now, we consider the original energy critical equation:
\EQn{\label{eq:nls-w-original-4d}
	\left\{ \aligned
	&i\pd_t \wt w + \De \wt w =  |\wt w|^2\wt w, \\
	& \wt w(0,x) = \wt w_0,
	\endaligned
	\right.
}
where $\wt w(t,x):\R\times \R^4\ra \C$. Let $g(t,x) := w(t,x)-\wt w(t,x)$. Then, the equation for $g$ is
\EQn{
	\label{eq:nls-w-difference-4d}
	\left\{ \aligned
	&i\pd_t g + \De g =  F(g+v,\wt w), \\
	& g(0,x) = w_0 - \wt w_0.
	\endaligned
	\right.
}
Here we denote that
\EQ{
	F(g,w):=|g+w|^2 (g+w) - |w|^2 w.
}
Recall that
\EQ{
	\norm{w}_{X(I)}=\norm{\jb{\nabla}w}_{L_t^2 L_x^{4}(I\times \R^4)} + \norm{w}_{L_t^4 L_x^8(I\times \R^4)} + \norm{w}_{L_{t,x}^4(I\times \R^4)},
}
and
Recall that
\EQ{
	\norm{v}_{Y(I)}:=&\normb{\jb{\nabla}^{s+a-}v}_{L_t^2 L_x^\I(I\times \R^4)} + \norm{v}_{L_t^4 L_x^8(I\times \R^4)} + \norm{v}_{L_t^6 L_x^3(I\times \R^4)} \\
	& + \normb{\jb{\nabla}^{-\frac14}v}_{L_{t,x}^4(I\times \R^4)}.
}
\begin{lem}\label{lem:short-time-4d}
	Let $a\in\N$, $a> 10$, $\frac12-\frac12a<s<0$, $I\subset \R$, and $0\in I$. Then, there exists $0<\eta_1\ll1$ with the following properties. Let $\wt w\in C\brkb{I;H_x^1(\R^4)}$ be the solution of \eqref{eq:nls-w-original} on $I$, satisfying
	\EQ{
		\norm{w_0-\wt w_0}_{\dot H_x^1(\R^3)}\loe \eta_1\text{, and } \norm{\wt w}_{X(I)}\loe \eta_1.
	}
	For any $0<\eta\loe\eta_1$, suppose that
	\EQ{
		\norm{v}_{Y(I)}\loe \eta,
	}
	then there exists a solution $w\in C\brkb{I; H_x^1(\R^4)}$ of \eqref{eq:nls-w} with initial data $w_0$ such that
	\EQ{
		\norm{w-\wt w}_{L_t^\I  H_x^1(I\times\R^4)} + \norm{w-\wt w}_{X(I)} \loe C_0\brkb{\norm{w_0-\wt w_0}_{ H_x^1(\R^4)} + \eta},
	}
	where $C_0>1$ is an absolute constant independent of $\eta$, $\eta_1$ and $I$.
\end{lem}
\begin{proof}
	Let $g$ be the solution of \eqref{eq:nls-w-difference}, $0=\inf I$, and we restrict the time interval on $I$. Then, we have
	\EQ{
		g=e^{it\De}g(0)-i\int_0^te^{i(t-s)\De}F(g+v,\wt w)\ds.
	}
	Note that we have the pointwise estimate
	\EQ{
		\abs{F(g+v,\wt w)} = & \abs{|g+v+\wt w|^2(g+v+\wt w)-|\wt w|^2\wt w}\\
		\lsm & (|g| + |v|)(|g|^2 + |v|^2 + |\wt w|^2 ),
	}
	and
	\EQ{
		\abs{\nabla F(g+v,\wt w)}=&\abs{\nabla \brk{|g+v+\wt w|^2(g+v+\wt w)-|\wt w|^2\wt w}}\\
		\lsm& \brk{\abs{\nabla g} + \abs{\nabla v}}\brk{|g|^2 + |v|^2 + |\wt w|^2} + \abs{\nabla \wt w}\brk{|g|^2 + |v|^2}\\
		& + \abs{\nabla \wt w}\brk{|g| + |v|}|\wt w|.
	}
Note that by the definition of $a$ and $s$, we have $s+a>1$. Therefore, since $v$ is high-frequency,
\EQ{
	\norm{v}_{L_t^2 L_x^\I}+ \norm{\nabla v}_{L_t^2 L_x^\I}\lsm \norm{v}_{Y}.
} 
	Then, using the similar argument in Lemma \ref{lem:local-4d},
	\EQ{
		\norm{F(g+v,\wt w)}_{L_t^1 L_x^2} \lsm & \norm{g}_{L_t^2 L_x^4}(\norm{g}_{L_t^4 L_x^{8}}^2 + \norm{v}_{L_t^4 L_x^{8}}^2 + \norm{\wt w}_{L_t^4 L_x^{8}}^2) \\
		& + \norm{v}_{L_t^2 L_x^\I}(\norm{g}_{L_{t,x}^4}^2 + \norm{v}_{L_{t,x}^4}^2 + \norm{\wt w}_{L_{t,x}^4}^2) \\
		\lsm & (\norm{g}_{X} + \norm{v}_{Y})(\norm{g}_{X}^2 + \norm{v}_{Y}^2 + \norm{\wt w}_{X}^2),
	}
	and
	\EQ{
		\norm{\nabla F(g+v,\wt w)}_{L_t^1 L_x^2} \lsm & \norm{\nabla g}_{L_t^2 L_x^4}(\norm{g}_{L_t^4 L_x^{8}}^2 + \norm{v}_{L_t^4 L_x^{8}}^2 + \norm{\wt w}_{L_t^4 L_x^{8}}^2) \\
		& + \norm{\nabla v}_{L_t^2 L_x^\I}(\norm{g}_{L_{t,x}^4}^2 + \norm{v}_{L_{t,x}^4}^2 + \norm{\wt w}_{L_{t,x}^4}^2) \\
		& +  \norm{\nabla \wt w}_{L_t^2 L_x^4}(\norm{g}_{L_t^4 L_x^{8}}^2 + \norm{v}_{L_t^4 L_x^{8}}^2) \\
		& + \norm{\nabla \wt w}_{L_t^2 L_x^4}(\norm{g}_{L_t^4 L_x^{8}} + \norm{v}_{L_t^4 L_x^{8}})\norm{\wt w}_{L_t^4 L_x^{8}} \\
		\lsm & (\norm{g}_{X} + \norm{v}_{Y})(\norm{g}_{X}^2 + \norm{v}_{Y}^2 + \norm{\wt w}_{X}^2) \\
		& + \norm{\wt w}_{X}(\norm{g}_{X}^2 + \norm{v}_{Y}^2 + \norm{g}_{X}\norm{\wt w}_{X} + \norm{v}_{Y}\norm{\wt w}_{X}).
	}
	Therefore, by Lemma \ref{lem:strichartz},
	\EQ{
		\norm{g}_{L_t^\I H_x^1 \cap X} \lsm & \norm{g(0)}_{ H_x^1} + \norm{\jb{\nabla} F(g+v,\wt w)}_{L_t^1 L_x^2}  \\
		\lsm & \norm{g(0)}_{ H_x^1} + (\norm{g}_{X} + \norm{v}_{Y})(\norm{g}_{X}^2 + \norm{v}_{Y}^2 + \norm{\wt w}_{X}^2) \\
		& + \norm{\wt w}_{X}(\norm{g}_{X}^2 + \norm{v}_{Y}^2 + \norm{g}_{X}\norm{\wt w}_{X} + \norm{v}_{Y}\norm{\wt w}_{X}) \\
		\lsm & \norm{g(0)}_{ H_x^1} + \brkb{\norm{g}_{X} + \eta}\brkb{\norm{g}_{X}^2 + \eta^2 + \eta_1^2} \\
		&+ \eta_1\brkb{\norm{g}_{X}^2 + \eta^2 + \norm{g}_{X}\eta_1 + \eta\eta_1} \\
		\lsm & \norm{g(0)}_{ H_x^1} + \norm{g}_{X}^3 + \eta_1^2\norm{g}_{X} + \eta\norm{g}_{X}^2 +\eta\eta_1^2 + \eta_1\norm{g}_{X}^2,
	}
	then we have
	\EQ{
		\norm{g}_{L_t^\I H_x^1 \cap S} \lsm & \norm{g(0)}_{ H_x^1} + \brkb{\norm{g}_{X}+\eta_1}\norm{g}_{X}^2 + \eta_1^2 \norm{g}_{X} + \eta_1^2\eta.
	}
	Then, this lemma follows by the standard continuity argument.
\end{proof}
\begin{lem}\label{lem:long-time-4d}
	Suppose that $a\in\N$, $a> 10$, $\frac12-\frac12a<s<0$, $M_0>1$, $I\subset\R$, and $0\in I$. Let $\wt w\in C\brkb{I;H_x^1(\R^4)}$ be the solution of \eqref{eq:nls-w-original} on $I$ with
	\EQ{
		\wt w_0=w_0\text{, and }\norm{\wt w}_{X(I)}\loe M_0.
	}
	Let $0<\eta_1\ll1$ and $C_0>1$ be defined as in Lemma \ref{lem:short-time-4d}. Then, there exists $\eta_2=\eta_2(C_0,M_0,\eta_1)>0$ such that if $v$ satisfies
	\EQ{
		\norm{v}_{Y(I)}\loe \eta_2,
	}
	then there exists a solution $w\in C\brkb{I;\dot H_x^1(\R^4)}$ of \eqref{eq:nls-w} with initial data $w_0$ such that
	\EQn{\label{eq:long-time-4d}
		\norm{w-\wt w}_{L_t^\I  H_x^1(I\times\R^4)} + \norm{w-\wt w}_{X(I)} \loe C(C_0,M_0,\eta_1)\eta_2.
	}
\end{lem}
The proof of Lemma \ref{lem:long-time-4d} is the same as Lemma \ref{lem:long-time}, so we omit the details.

\subsection{Proof of Proposition \ref{prop:global-derterministic-4d}}\label{sec:global-scattering-4d}
Using the classical result in \cite{RV07AJM}, we can obtain
\begin{lem}\label{lem:energy-critical-classical-4d}
	Suppose that $\wt w_0 \in H^1(\R^4)$. Then, the equation \eqref{eq:nls-w-original} is globally well-posed and scatters, and the solution $\wt w\in C\brko{\R; H^1(\R^4)}$ satisfies
	\EQ{
		\norm{\wt w}_{X(\R)} \loe C(\norm{\wt w_0}_{ H_x^1(\R^3)}).
	}
\end{lem}
Then, we are able to prove the global well-posedness of $w\in C(\R;H_x^1(\R^4))$, using the same argument in Section \ref{sec:global-scattering}. Moreover, we have
\EQ{
\norm{w}_{L_t^\I(\R;H_x^1(\R^4))}+\norm{w}_{X(\R)} \loe C(A).
}

Next, we prove the scattering statement. We only consider the $t\ra+\I$ case, and it suffices to prove that
\EQn{\label{eq:scattering-space-time-4d}
	\normb{\jb{\nabla} \int_0^\I e^{-is\De}(|u|^2u)\dx}_{L_x^2}\loe C(A).
}
Now, all the space-time norms are taken over $[0,+\I)\times \R^4$. From previous argument,  
\EQ{
	\norm{w}_{X([0,+\I))}\loe C(A).
}
Recall also that 
\EQ{
	\norm{v}_{Y(\R)}\loe A.
}
Now, we can prove \eqref{eq:scattering-space-time-4d} using the argument in Lemma \ref{lem:local-4d},
\EQ{
	\text{L.H.S. of }\eqref{eq:scattering-space-time-4d} \lsm & \normb{ \int_0^\I e^{-is\De}( |u|^2u)\ds}_{L_x^2} + \normb{ \int_0^\I e^{-is\De}(\nabla wu^2)\ds}_{L_x^2} \\
	& + \normb{ \int_0^\I e^{-is\De}(\nabla v u^2)\ds}_{L_x^2}\\
	\lsm & \norm{|u|^2u}_{L_t^1 L_x^2} + \norm{\nabla wu^2}_{L_t^1 L_x^2} + \norm{\nabla vu^2}_{L_t^1 L_x^2}\\
	\lsm & \norm{\jb{\nabla}w}_{L_t^2 L_x^4}\norm{u}_{L_t^4 L_x^{8}}^2 + \norm{\jb{\nabla}v}_{L_t^2 L_x^\I}\norm{u}_{L_{t,x}^4}^2\\
	\loe & C(A).
}
This finishes the proof of scattering statement.

\end{document}